\documentclass[reqno]{amsart}

\usepackage[all]{xy}

\usepackage{graphics}
\usepackage[usenames]{color}

\usepackage[centertags]{amsmath}
\usepackage{amsfonts}
\usepackage{amssymb}
\usepackage{amscd}
\usepackage{amsthm}
\usepackage{setspace}
\usepackage{graphicx}
\usepackage{stmaryrd}
\usepackage[all]{xy}
\usepackage[english]{babel}
\usepackage[latin1]{inputenc}
\usepackage{epsf}

\newcommand{\st} [1]     {\scriptscriptstyle{{#1}}}
\newcommand{\rmap}       {\longrightarrow}

\newcommand{\Ker}        {{\mathrm {ker}}}

\newcommand{\Cour}[1]      {[\![#1]\!]}

\newcommand{\Lie}        {\mathcal L}

\theoremstyle{plain}
\newtheorem{theorem}{Theorem}[section]
\newtheorem{proposition}{Proposition}[section]
\newtheorem{lemma}{Lemma}[section]
\newtheorem{corollary}{Corollary}[section]
\theoremstyle{definition}
\newtheorem{example}{Example}[section]
\newtheorem{definition}{Definition}[section]
\newtheorem{remark}{Remark}[section]

\newcommand\R{\mathbb{R}}

\newcommand{\C}{\mathbb C}

\DeclareMathAlphabet{\mathpzc}{OT1}{pcz}{m}{it}

\begin{document}
\title{Multiplicative Dirac Structures}

\author{Cristi\'an Ortiz}
\address{Departamento de Matem\'atica\\
Universidade Federal do Paran\'a\\
Setor de Ci\^encias Exatas - Centro Polit\'ecnico
81531-990 Curitiba - Brasil.}
\email{cristian.ortiz@ufpr.br}
\thanks{}

\date{}


\begin{abstract}

In this paper we introduce multiplicative Dirac structures on Lie groupoids, providing a unified framework to study both multiplicative Poisson bivectors (i.e., Poisson group(oid)s) and multiplicative closed $2$-forms (e.g., symplectic groupoids). We prove that for every source simply connected Lie groupoid $G$ with Lie algebroid $AG$, there exists a one-to-one correspondence between multiplicative Dirac structures on $G$ and Dirac structures on $AG$, which are compatible with both the linear and algebroid structures of $AG$. We explain in what sense this extends the integration of Lie bialgebroids to Poisson groupoids carried out in \cite{MX2} and the integration of Dirac manifolds of \cite{BCWZ}. We also explain the connection between multiplicative Dirac structures and higher geometric structures such as $\mathcal{LA}$-groupoids and $\mathcal{CA}$-groupoids.

\end{abstract}

\maketitle

\tableofcontents

\section{Introduction}

Dirac structures were introduced by Courant and Weinstein \cite{CW} as a common generalization of Poisson bivectors, closed $2$-forms and regular foliations. A \textbf{Dirac structure} on a smooth manifold $M$ consists on a vector subbundle $L\subseteq \mathbb{T}M:=TM\oplus T^*M$, which is maximal isotropic with respect to the nondegenerate symmetric pairing on $\mathbb{T}M,$

$$
\langle (X,\alpha),(Y,\beta)\rangle= \alpha(Y)+ \beta(X),
$$
and that satisfies the integrability condition

$$\Cour{\Gamma(L),\Gamma(L)}\subseteq \Gamma(L),$$

\noindent with respect to the Courant bracket $\Cour{\cdot,\cdot}:\Gamma(\mathbb{T}M)\times\Gamma(\mathbb{T}M)\rmap \Gamma(\mathbb{T}M),$

$$\Cour{(X,\alpha),(Y,\beta)}=([X,Y], \Lie_{X}\beta- i_{Y}d\alpha).$$

\noindent The integrability in the sense of Courant unifies different integrability conditions, including closed $2$-forms, Poisson bivectors and regular foliations (see \cite{C, CW}). More precisely, a $2$-form $\omega$ on a smooth manifold $M$ induces a bundle map $\omega^{\sharp}:TM\rmap T^*M, X\mapsto \omega(X,\cdot)$ whose graph $L_{\omega}=\{(X,\omega	^{\sharp}(X)\mid X\in TM)\}$ is a Lagrangian subbundle of $\mathbb{T}M$. In this case, the Courant integrability of $L_{\omega}$ is equivalent to $\omega$ being a closed $2$-form. Similarly, any bivector $\pi$ on $M$ defines a bundle map $\pi^{\sharp}:T^*M\rmap TM, \alpha\mapsto \pi(\alpha,\cdot)$ whose graph $L_{\pi}=\{(\pi^{\sharp}(\alpha),\alpha)\}$ is a Lagrangian subbundle of $\mathbb{T}M$. One checks that $L_{\pi}$ satisfies the Courant integrability condition if and only if $\pi$ is a Poisson bivector. Also, if $F\subseteq TM$ is a regular subbundle we denote by $F^{\circ}\subseteq T^*M$ its annihilator. Then the Lagrangian subbundle $F\oplus F^{\circ}\subseteq \mathbb{T}M$ is integrable in the sense of Courant if and only if $F\subseteq TM$ is involutive with respect to the Lie bracket of vector fields.

The main objective of this paper is to study Dirac structures defined on Lie groupoids, satisfying a suitable compatibility condition with the groupoid multiplication.  Our study is motivated by a variety of geometrical structures compatible with group or groupoid structures, including:

\begin{itemize}
\item[i)] \textit{Poisson-Lie groups}: these structures consist of a Lie group $G$ with a Poisson structure $\pi$, which are compatible in the sense that the multiplication map $m:G\times G\rmap G$ is a Poisson map. Equivalently, the Poisson bivector $\pi$ is \textbf{multiplicative}, that is

$$\pi_{gh}=(l_g)_*\pi_h+ (r_h)_*\pi_g,$$

\noindent for every $g,h\in G$. Here $l_g$ and $r_h$ denote the left and right multiplication by $g$ and $h$, respectively. Poisson-Lie groups arise as semiclassical limit of quantum groups, and they are infinitesimally described by \textit{Lie bialgebras}. See e.g. \cite{D}.

\item[ii)] \textit{Symplectic groupoids}: a symplectic groupoid is a Lie groupoid $G$ with a symplectic structure
$\omega$, which is compatible with the groupoid multiplication in the sense that the graph

$$\mathrm{Graph}(m)\subseteq G\times G\times \overline{G}$$

\noindent is a Lagrangian submanifold with respect to the symplectic structure $\omega\oplus\omega\ominus\omega$. This compatibility condition is equivalent to saying that $\omega$ is \textbf{multiplicative}, that is

$$m^*\omega=pr^*_1\omega+pr^*_2\omega,$$

\noindent where $pr_1,pr_2:G_{(2)}\rmap G$ are the canonical projections and $G_{(2)}\subseteq G\times G$ is the set of composable groupoid pairs. Symplectic groupoids arise in the context of quantization of Poisson manifolds \cite{Wsympgpds, WXU}, connecting Poisson geometry to noncommutative geometry. In \cite{CF}, symplectic groupoids appeared as phase spaces of certain sigma models. The infinitesimal description of symplectic groupoids is given by \emph{Poisson structures}, see e.g. \cite{Wsympgpds, CDW}.

\item[iii)] \textit{Poisson groupoids}: these objects were introduced by A. Weinstein \cite{Wco} unifying Poisson-Lie groups and symplectic groupoids. A Poisson groupoid is a Lie groupoid $G$ equipped with a Poisson structure $\pi$, which is compatible with the groupoid multiplication in the sense that

$$\mathrm{Graph}(m)\subseteq G\times G\times \overline{G}$$

\noindent is a coisotropic submanifold. These structures are related to the geometry of the classical dynamical Yang-Baxter equation, see for instance \cite{EV}. At the infinitesimal level, Poisson groupoids are described by \textit{Lie bialgebroids} \cite{MX1}.

\item[iv)] \textit{Presymplectic groupoids}: Lie groupoids equipped with a multiplicative closed $2$-form were studied in \cite{BCWZ}. A presymplectic groupoid \cite{BCWZ} is a Lie groupoid $G$ with a multiplicative closed $2$-form $\omega$ satisfying suitable nondegeneracy conditions. These objects arise in connection with equivariant cohomology and generalized moment maps \cite{BC2}. The infinitesimal description of presymplectic groupoids is given by \emph{Dirac structures}, extending the infinitesimal description of symplectic groupoids. More generally, Lie groupoids endowed with arbitrary multiplicative closed $2$-forms are infinitesimally described by bundle maps $\sigma:AG\rmap T^*M$ called \emph{IM-$2$-forms}. Here $AG$ denotes the Lie algebroid of $G$ and $T^*M$ is the cotangent bundle of the base of $G$.

\end{itemize}

The first goal of this work is to find a suitable definition of multiplicative Dirac structure that include
both multiplicative Poisson bivectors and multiplicative closed 2-forms, and hence encompasses all examples above. This is obtained by observing that given a Lie groupoid $G$ over $M$ with Lie algebroid $AG$, the tangent bundle $TG$ and the cotangent bundle $T^*G$ inherit natural Lie groupoid structures over $TM$ and $A^*G$, respectively. One observes that a bivector $\pi_G$ is multiplicative if and only if the bundle map $\pi^{\sharp}_G:T^*G\rmap TG$ is a groupoid morphism \cite{MX1}. Similarly, a $2$-form $\omega_G$ is multiplicative if and only if the bundle map $\omega^{\sharp}_G:TG\rmap T^*G$ is a morphism of Lie groupoids. It turns out that the direct sum vector bundle $TG\oplus T^*G$ is a Lie groupoid over $TM\oplus A^*G$, and graphs of both multiplicative Poisson bivectors and multiplicative closed $2$-forms define Lie subgroupoids of $TG\oplus T^*G$. We say that a Dirac structure $L_G$ on a Lie groupoid $G$ is \textbf{multiplicative} if $L_G\subseteq TG\oplus T^*G$ is a Lie subgroupoid. A Lie groupoid $G$ equipped with a multiplicative Dirac structure is referred to as a \textbf{Dirac groupoid}.

Our main purpose is to describe multiplicative Dirac structures infinitesimally, that is, in terms of Lie algebroid data. This work can be considered as a first step toward such a description. The main result of the present work says that for every source simply connected Lie groupoid $G$ with Lie algebroid $AG$, multiplicative Dirac structures on $G$ correspond to Dirac structures on $AG$ suitably compatible with both the linear and Lie algebroid structures on $AG$. In the particular case of multiplicative Poisson bivectors and multiplicative $2$-forms, we explain how this is equivalent to the known infinitesimal descriptions carried out in \cite{MX2} and \cite{BCWZ}, respectively. Along the way, we develop techniques that can treat all multiplicative structures above in a unified manner, often simplifying existing results and proofs.

The present paper is organized as follows. In section 2 we recall the definition of tangent and cotangent structures including: tangent and cotangent groupoids and their algebroids, i.e. tangent and cotangent algebroids. We also give an intrinsic construction of the tangent lift of a Dirac structure, providing an alternative proof of the results shown in \cite{Ctangent}. In section 3 we define the main objects of our study, multiplicative Dirac structures. We discuss a variety of examples arising in nature, including: foliated groupoids, Dirac Lie groups, tangent lifts of multiplicative Dirac structures, symmetries of multiplicative Dirac structures (e.g. reduction of Poisson groupoids), $B$-field transformations of multiplicative Dirac structures and generalized complex groupoids. In section 4 we introduce the notion of Dirac algebroid and also several examples are discussed, including: foliated algebroids, Dirac Lie algebras, tangent lifts of Dirac algebroids, symmetries of Dirac algebroids (e.g. reduction of Lie bialgebroids), $B$-field transformations of Dirac algebroids and generalized complex algebroids. In section 5 we explain how the multiplicativity of a Dirac structure is reflected at the Lie algebroid level, proving the main result of this work, which says that if $G$ is a source simply connected Lie groupoid with Lie algebroid $AG$, then there is a one-to-one correspondence between Dirac groupoid structures on $G$ and Dirac algebroid structures on $AG$. Along the way, we explain the relation between multiplicative Dirac structures and higher structures such as $\mathcal{CA}$-groupoids and $\mathcal{LA}$-groupoids. We also relate the examples of section 3 with the examples of section 4, in the spirit of the correspondence established by the main result of the paper. In section 6, we discuss conclusions and work in progress.

\subsection{Acknowledgements} The content of this work is based on the author's doctoral thesis \cite{OrtizThesis} carried out at IMPA under the supervision of Henrique Bursztyn. The author would like to thank Henrique Bursztyn for suggesting the problem of studying multiplicative Dirac structures as well as for several conversations and suggestions that have improved this work. While the author was at IMPA, he was supported by a PEC-PG doctoral scholarship from CAPES (Brasil).


\subsection{Notation and conventions}

For a Lie groupoid $G$ over $M$ we denote by $s,t:G\rmap M$ the source and target maps, respectively. The multiplication map is denoted by $m:G_{(2)}\rmap G$, where $G_{(2)}=\{(g,h)\in G\times G\mid s(g)=t(h)\}$ is the set of composable pairs. The Lie algebroid of $G$ is defined by $AG:=\mathrm{Ker}(Ts)|_{M}$, with Lie bracket given by identifying sections of $AG$ with right-invariant vector fields on $G$ and anchor map $\rho_{AG}:=Tt|_{AG}:AG\rmap TM$. Given a Lie groupoid morphism $\Psi:G_1\rmap G_2$, the corresponding Lie algebroid morphism is denoted by $A(\Psi):AG_1\rmap AG_2$. Arbitrary Lie algebroids are denoted by $A\rmap M$ with Lie bracket $[\cdot,\cdot]_A$ and anchor map $\rho_A$. Also, given a smooth manifold $M$, the tangent bundle is denoted by $p_M:TM\rmap M$ and the cotangent bundle is denoted by $c_M:T^*M\rmap M$.


\section{Tangent and cotangent structures}

\subsection{Tangent and cotangent groupoids}\label{tangentgroupoids}

Let $G$ be a Lie groupoid over $M$ with Lie algebroid $AG$. The tangent bundle $TG$ has a natural Lie groupoid structure over $TM$. This structure is obtained by applying the tangent functor to each of the structure maps defining $G$ (source, target, multiplication, inversion and identity section). We refer to $TG$ with this groupoid structure over $TM$ as the \textbf{tangent groupoid} of $G$. Notice that the set of composable pairs $(TG)_{(2)}=T(G_{(2)})$, and for $(g,h)\in G_{(2)}$ and a tangent groupoid pair $(X_g,Y_h)\in (TG)_{(2)}$ the multiplication map on $TG$ is 

$$X_g\bullet Y_h:=Tm(X_g,Y_h)$$

Consider now the cotangent bundle $T^*G$. It was shown in \cite{CDW}, that $T^*G$ is a Lie groupoid over $A^*G$. The source and target maps are defined by

$$\tilde{s}(\alpha_g)u=\alpha_g(Tl_g(u-Tt(u)))\quad \text{ and } \tilde{t}(\beta_g)v=\beta_g(Tr_g(v))$$

\noindent where $\alpha_g \in T^*_gG$, $u\in A_{s(g)}G$ and $\beta_g\in T^*_gG$, $v\in A_{t(g)}G$. The multiplication on $T^*G$ is defined by

$$(\alpha_g\circ \beta_h)(X_g\bullet Y_h)= \alpha_g(X_g)+ \beta_h(Y_h)$$

\noindent for $(X_g,Y_h)\in T_{(g,h)}G_{(2)}$.

We refer to $T^*G$ with the groupoid structure over $A^*G$ as the \textbf{cotangent groupoid} of $G$.


\subsection{Tangent and cotangent algebroids}\label{tangentalgebroids}

Let $q_A:A\rmap M$ be a \emph{vector bundle} over $M$. The tangent bundle $TA$ has a natural structure of a \emph{double vector bundle} \cite{Pradines}, given by the diagram below.

\begin{align}
\xy 
(-15,10)*+{TA}="t0"; (-15,-10)*+{A}="b0"; 
(15,10)*+{TM}="t1"; (15,-10)*+{M}="b1"; 
{\ar@<.25ex>^{Tq_A} "t0"; "t1"}; 
{\ar@<.25ex>_{q_A} "b0"; "b1"}; 
{\ar@<.25ex>_{p_A} "t0"; "b0"}; 
{\ar@<.25ex>^{p_M} "t1"; "b1"}; 
\endxy
\end{align}

Assume now that $q_A:A\rmap M$ has a Lie algebroid structure with anchor map $\rho_A:A\rmap TM$ and Lie bracket $[\cdot,\cdot]$ on $\Gamma_M(A)$.

As explained in \cite{M}, there is a canonical Lie algebroid structure on the vector bundle $Tq_A:TA\rmap TM$. Recall that there exists a \textbf{canonical involution} $J_M:TTM\rmap TTM$ which is a morphism of double vector bundles. In a local coordinates system $(x^i,\dot{x}^i,\delta x^i,\delta\dot{x}^i)$ on $TTM$ this map is given by $$J_M((x^i,\dot{x}^i,\delta x^i,\delta\dot{x}^i))=(x^i,\delta x^i,\dot{x}^i,\delta\dot{x}^i).$$


\noindent Now we can apply the tangent functor to the anchor map $\rho_A:A\rmap TM$, and then compose with the canonical involution to obtain a bundle map $\rho_{TA}:TA\rmap TTM$ defined by

$$\rho_{TA}:=J_M\circ T\rho_{A}.$$ 

\noindent This defines the tangent anchor map. In order to define the tangent Lie bracket, we observe that every section $u\in\Gamma_{M}(A)$ induces two types of sections of $Tq_A:TA\rmap TM$. The first type corresponds to the \textbf{linear section}$Tu:TM\rmap TA$, which is given by applying the tangent functor to the section $u:M\rmap A$. The second type of section is the \textbf{core} section $\hat{u}:TM\rmap TA$, which is defined by

$$\hat{u}(X)=T(0^A)(X)+_A\overline{u(p_M(X))},$$   

\noindent where $0^A:M\rmap A$ denotes the zero section, and $\overline{u(p_M(X))}=\frac{d}{dt}(tu(p_M(X)))\vert_{t=0}$. As observed in \cite{MX1}, sections of the form $Tu$ and $\hat{u}$ generate the module of sections $\Gamma_{TM}(TA)$. Therefore, the tangent Lie bracket is determined by

$$[Tu,Tv]=T[u,v], \quad [Tu,\hat{v}]=\widehat{[u,v]}, \quad[\hat{u},\hat{v}]=0,$$

\noindent and we extend to other sections by requiring the Leibniz rule with respect to the tangent anchor $\rho_{TA}$.

Following \cite{M}, the cotangent bundle of a Lie algebroid inherits a Lie algebroid structure. For that, let us explain the vector bundle structure $T^*A\rmap A^*$. If $(x^i,u^a)$ are local coordinates on $A$, we induce a local coordinates system $(x^i,u^a,p_i,\lambda_a)$ on $T^*A$, where $(p_i)$ determines a cotangent element in $T^{*}_{x}M$ and $(\lambda_a)\in A^*_{x}$ is a cotangent element with respect to the tangent direction to the fibers of $A$. Now the bundle projection $r:T^*A\rmap A^*$ is described locally by $r(x^i,u^a,p_i,\lambda_a)=(x^i,\lambda_a)$. These vector bundle structures define a commutative diagram

\begin{align}
\xy 
(-15,10)*+{T^*A}="t0"; (-15,-10)*+{A}="b0"; 
(15,10)*+{A^*}="t1"; (15,-10)*+{M}="b1"; 
{\ar@<.25ex>^{r} "t0"; "t1"}; 
{\ar@<.25ex>_{q_{A}} "b0"; "b1"}; 
{\ar@<.25ex>_{c_{A}} "t0"; "b0"}; 
{\ar@<.25ex>^{q_{A^*}} "t1"; "b1"}; 
\endxy
\end{align}

\noindent This endows $T^*A$ with a double vector bundle structure. Suppose that $q_A:A\rmap M$ carries a Lie algebroid structure. Then we can consider the dual bundle $A^*$ endowed with the linear Poisson structure induced by $A$. The cotangent bundle $T^*A^*\rmap A^*$ has the Lie algebroid structure determined by the linear Poisson bivector on $A^*$. There exists a Legendre type map $R:T^*A^*\rmap T^*A$ which is a anti-symplectomorphism with respect to the canonical symplectic structures, and it is locally defined by $R(x^i,\xi_a,p_i,u^a)=(x^i,u^a,-p_i,\xi_a)$. For an intrinsic definition see \cite{MX1,Tulc}.

\begin{definition}
 The \textbf{cotangent algebroid} of $A$ is the vector bundle $T^*A\rmap A^*$ equipped with the unique Lie algebroid structure that makes the Legendre type transform $R:T^*(A^*)\rmap T^*A$ into an isomorphism of Lie algebroids.
\end{definition}

Finally, recall also that the \textbf{Tulczyjew map} $\Theta_M:TT^*M\rmap T^*TM$ is the isomorphism, which in a local coordinates system $(x^i,p_i,\dot{x}^i,\dot{p}_i)$ is given by

$$\Theta_M(x^i,p_i,\dot{x}^i,\dot{p}_i)=(x^i,\dot{x}^i,\dot{p}_i,p_i).$$

See \cite{MX1,Tulc} for an intrinsic definition. Consider now a Lie groupoid $G$ over $M$ with Lie algebroid $AG$. There exists a natural injective bundle map

\begin{equation}\label{i_{AG}}
i_{AG}:AG\rmap TG
\end{equation}

\noindent The canonical involution $J_G:TTG\rmap TTG$ restricts to an isomorphism of Lie algebroids $j_G:T(AG)\rmap A(TG)$. More precisely, there exists a commutative diagram

\begin{align}\label{ATG}
\xy 
(-15,10)*+{T(AG)}="t0"; (-15,-10)*+{TTG}="b0"; 
(15,10)*+{A(TG)}="t1"; (15,-10)*+{TTG}="b1"; 
{\ar@<.25ex>^{j_G} "t0"; "t1"}; 
{\ar@<.25ex>_{J_G} "b0"; "b1"}; 
{\ar@<.25ex>_{T(i_{AG})} "t0"; "b0"}; 
{\ar@<.25ex>^{i_{A(TG)}} "t1"; "b1"}; 
\endxy
\end{align}

\noindent In particular, the Lie algebroid $A(TG)$ of the tangent groupoid is canonically isomorphic to the tangent Lie algebroid $T(AG)$ of $AG$. Similarly, the Lie algebroid of the cotangent groupoid $T^*G$ is isomorphic to the cotangent Lie algebroid $T^*(AG)$. For that, notice that the natural pairing $T^*G\oplus TG\rmap \R$ defines a groupoid morphism, and the application of the Lie functor yields a symmetric pairing $\langle\langle\cdot,\cdot\rangle\rangle:A(T^*G)\oplus A(TG)\rmap \R$, which is nondegenerate. See e.g. \cite{MX1,MX2}. In particular, we obtain an isomorphism $K_G:A(T^*G)\rmap A(TG)^*$, where the target dual is with respect to the fibration $A(TG)\stackrel{A(p_G)}{\longrightarrow} AG$. Now we define a Lie algebroid isomorphism

\begin{equation}\label{AT*G}
j'_G:A(T^*G)\rmap T^*(AG),
\end{equation}

\noindent determined by the composition $j'_G=j^*_{G}\circ K_G$, where $j^*_G: A(TG)^* \rmap T^*(AG)$ is the bundle map dual to the isomorphism $j_G:T(AG) \rmap A(TG)$. As $j_G:T(AG) \rmap A(TG)$ is a suitable restriction of the canonical involution $J_G:TTG \rmap TTG$, the isomorphism $j'_G$ is related to the Tulczyjew map $\Theta_G:TT^*G\rmap T^*TG$, via

$$j'_G=(Ti_{AG})^*\circ \Theta_G\circ i_{A(T^*G)}.$$


\subsection{Tangent lift of a Dirac structure}

The tangent lift of Dirac structures was originally studied by T. Courant \cite{Ctangent}, where tangent Dirac structures are described locally. In \cite{Vaisman} I. Vaisman gives an intrinsic construction of tangent Dirac structures, where the tangent lift of a Dirac structure is described via the sheaf of local sections defining a Dirac subbundle of $TTM\oplus T^*TM$. Here, we provide an alternative description of the tangent lift of a Dirac structure relied on the tangent lift of Lie algebroid structures described in the previous section. 

In order to fix our notation, we begin by summarizing some of the main properties of tangent lifts of vector fields and differential forms , see \cite{Grabowski, Yano}. Let $f\in C^{\infty}(M)$ be a smooth function. Then we have a pair of smooth functions on $TM$ defined by

$$f^v=f\circ p_M; \quad f^{T}=df.$$

\noindent We refer to $f^v$ and $f^{T}$ as the \textbf{vertical} lift and \textbf{tangent} lift of $f$, respectively. One can see easily that the algebra of functions $C^{\infty}(TM)$ is generated by functions of the form $f^v$ and $f^{T}$. Now, given a vector field $X$ on $M$ we define the \textbf{vertical} lift of $X$ as the vector field $X^v$ on $TM$ which acts on vertical and tangent lifts of functions as

$$X^v(f^v)=0, \quad X^v(f^{T})=(Xf)^v.$$

The \textbf{tangent} lift of $X$ is the vector field $X^{T}$ on $TM$, which acts on vertical and tangent lifts of functions in the following manner:

$$X^{T}(f^v)=(Xf)^v, \quad X^{T}(f^{T})=(Xf)^{T}.$$

\noindent It is easy to see that vertical and tangent lifts of vector fields generate the space of all vector fields on $TM$. Now let us consider a $1$-form $\alpha$ on a smooth manifold $M$. We define the \textbf{vertical} lift of $\alpha$ as the $1$-form $\alpha^v$ on $TM$, which is determined by its value at vertical and tangent lifts of vector fields,

$$\alpha^v(X^v)=0, \quad \alpha^v(X^{T})=(\alpha(X))^v.$$

\noindent The \textbf{tangent} lift of $\alpha$ is the $1$-form $\alpha^{T}$ on $TM$ defined by

$$\alpha^{T}(X^v)=(\alpha(X))^v, \quad \alpha^{T}(X^{T})=(\alpha(X))^{T}.$$

It is important to emphasize that vertical and tangent lifts of vector fields (resp. of $1$-forms) are sections of the usual vector bundle structure $T(TM)\stackrel{p_{TM}}{\longrightarrow}TM$ (resp. sections of $T^*(TM)\stackrel{c_{TM}}{\longrightarrow}TM$), and they do not define sections of the tangent prolongation vector bundle $T(TM)\stackrel{Tp_M}{\longrightarrow}TM$ (resp. of the tangent prolongation $T(T^*M)\stackrel{Tc_M}{\longrightarrow}TM$). However, there exists a canonical relation between vector fields (resp. $1$-forms) on $TM$ and sections of the tangent prolongation vector bundle $T(TM)\rmap TM$ (resp. $T(T^*M)\rmap TM$). Given a vector field $X$ and a $1$-form $\alpha$ on $M$, we consider the linear sections $TX, T\alpha$ and the core sections $\hat{X},\hat{\alpha}$ of the corresponding tangent prolongation vector bundles. It follows from the definition that

\begin{align}
 &J_M(TX)=X^T,\quad J_M(\hat{X})=X^v.\label{JMrelations}\\ 
&\Theta_M(T\alpha)=\alpha^T,\quad \Theta_{M}(\hat{\alpha})=\alpha^v.\label{thetarelations}
\end{align}

\noindent It turns out that many geometric properties of the direct sum vector bundle $T(TM)\oplus T^*(TM)$ can be understood in terms of tangent geometric properties of $T(TM)\oplus T(T^*M)$, using the canonical identification 

$$J_M\oplus \Theta_M:T(TM)\oplus T(T^*M)\rmap T(TM)\oplus T^*(TM).$$

Consider now a Dirac structure $L_M$ on $M$. Equivalently, we may think of $L_M$ as a Lie algebroid over $M$ with Lie bracket given by the Courant bracket on sections of $L_M$, and the anchor map $\rho_M$ is the natural projection from $L_M\subseteq TM\oplus T^*M$ onto $TM$. According to a construction of K. Mackenzie and P. Xu \cite{MX1}, we can consider the tangent prolongation Lie algebroid $TL_M\rmap TM$, with anchor map

$$\rho_{TM}=J_M\circ T\rho_M,$$

\noindent and Lie bracket defined by

$$[\hat{a}_1,\hat{a}_2]_{TL_M}=0, \quad [Ta_1,\hat{a}_2]_{TL_M}=\widehat{[a_1,a_2]}, \quad [Ta_1,Ta_2]_{TL_M}=T[a_1,a_2],$$

\noindent where $a_1,a_2$ are sections of $L_M\rmap M$. We denote by $L_{TM}$ the image of $TL_M$ under the natural bundle map $J_M\oplus \Theta_M:TTM\oplus TT^*M\rmap TTM\oplus T^*TM$. 

\begin{proposition}\label{tangentisotropic}

The subbundle $L_{TM}\subseteq TTM\oplus T^*TM$ is isotropic with respect to the non degenerate symmetric pairing $\langle\cdot,\cdot\rangle_{TM}$ defined on $TTM\oplus T^*TM$. 

\end{proposition}

\begin{proof}

Consider the non degenerate symmetric pairing $\langle\cdot,\cdot\rangle_{M}$ defined on $TM\oplus T^*M$. The application of the tangent functor, followed by the projection onto de second factor, leads to a non degenerate symmetric pairing

 $$\langle\langle\cdot,\cdot\rangle\rangle:TTM\times_{TM}TT^*M\rmap \R,$$ 

\noindent for which the subbundle $TL_M\subseteq TTM\oplus TT^*M$ is isotropic. Finally, for every $\dot{a}_1,\dot{a}_2\in TL_M$ the well known identity

$$\langle\langle \dot{a}_1,\dot{a}_2\rangle\rangle=\langle (J_M\oplus \Theta_M)(\dot{a}_1),(J_M\oplus \Theta_M)(\dot{a}_2)\rangle_{TM},$$ 

\noindent says that the canonical map $J_M\oplus \Theta_M:T(TM)\oplus T(T^*M)\rmap T(TM)\oplus T^*(TM)$ is a fiberwise isometry with respect to the pairings $\langle\langle\cdot,\cdot\rangle\rangle$ and $\langle\cdot,\cdot\rangle_{TM}$; see for instance \cite{Grabowski, MX1}. In particular, $L_{TM}=(J_M\oplus\Theta_M)(TL_M)$ is isotropic with respect to the canonical pairing on $TTM\oplus T^*TM$.

\end{proof}

The tangent Lie algebroid $TL_M\rmap TM$ induces a unique Lie algebroid structure on $L_{TM}\rmap TM$ characterized by the property that $J_M\oplus \Theta_M:TL_M\rmap L_{TM}$ is a Lie algebroid isomorphism. The space of sections $\Gamma(L_{TM})$ is generated by sections of the form $a^T:=(J_M\oplus\Theta_M)(Ta)$ and $a^{v}:=(J_M\oplus\Theta_M)\hat{a}$, where $a$ is a section of $L_M\rmap M$. In particular the induced Lie bracket on sections of $L_{TM}$ is completely determined by identities 

$$[a^v_1,a^v_2]=0, \quad [a^{T}_1,a^{v}_2]=\Cour{a_1,a_2}^{v}, \quad [a^{T}_1,a^{T}_2]=\Cour{a_1,a_2}^{T},$$
\noindent and the Leibniz rule with respect to the induced anchor map $pr_{TTM}:L_{TM}\rmap TTM$.

\begin{proposition}\label{tangentintegrability}

The induced Lie bracket on sections $\Gamma(L_{TM})$ is a restriction of the Courant bracket $\Cour{\cdot,\cdot}_{TM} $ on sections of $TTM\oplus T^*TM$.

\end{proposition}

\begin{proof}
 
Due to the identities (\ref{JMrelations}) and (\ref{thetarelations}), we only need to check that the Courant bracket on sections of $L_{TM}$, naturally induced by $J_M\oplus \Theta_M$, satisfies the bracket identities that determine the induced Lie bracket on $\Gamma(L_{TM})$. One observes that vertical and tangent lifts are compatible with Lie derivatives in the sense that

\begin{enumerate}
 \item $\Lie_{X^v}\alpha^v=0$

 \item $\Lie_{X^T}\alpha^v=(\Lie_X\alpha)^v$

 \item $\Lie_{X^T}\alpha^T=(\Lie_X\alpha)^T,$

\end{enumerate}

\noindent and we conclude that

\begin{enumerate}
 \item $\Cour{X^v\oplus \alpha^v,Y^v\oplus \beta^v}=0$

 \item $\Cour{X^T\oplus \alpha^T,Y^v\oplus \beta^v}=[X,Y]^v\oplus (\Lie_X\beta -i_{Y}d\alpha)^v$

\item$\Cour{X^T\oplus \alpha^T,Y^T\oplus \beta^T}=[X,Y]^T\oplus(\Lie_X\beta -i_{Y}d\alpha)^T$.

\end{enumerate}

Thus the Lie bracket on $\Gamma_{TM}(L_{TM})$ induced by the tangent Lie bracket on $\Gamma_{TM}(TL_M)$ coincides with the Courant bracket.

\end{proof}

We have shown the following.

\begin{proposition}
Let $M$ be a smooth manifold. There exists a natural map 

\begin{align*}
\rm{Dir}(&M)\rmap \mathrm{Dir}(TM)\\
&L_M\mapsto L_{TM}, 
\end{align*}

\noindent where $L_{TM}:=(J_M\oplus\Theta_M)(TL_M)$.
\end{proposition}

The Dirac structure $L_{TM}\in \mathrm{Dir}(TM)$ given by the proposition above is referred to as the \textbf{tangent Dirac structure} induced by $L_M\in\mathrm{Dir}(M)$. It is straightforward to check that this construction unifies the tangent lift of both closed $2$-forms and Poisson bivectors. Additionally, the presymplectic foliation of $L_{TM}$ corresponds to taking the tangent bundle of each leaf endowed with the tangent lift of the leafwise presymplectic forms defined by $L_M$.



\section{Multiplicative Dirac structures}\label{DiracGroupoids}

This section introduces the main objects of study of this work, that is, Lie groupoids equipped with Dirac structures compatible with the groupoid multiplication, including both multiplicative Poisson and closed $2$-forms as particular cases.

\subsection{Definition and main examples}

Let $G$ be a Lie groupoid over $M$, with Lie algebroid $AG$. Consider the direct sum Lie groupoid $\mathbb{T}G=TG\oplus T^*G$ with base manifold $TM\oplus A^*G$.

\begin{definition}\label{def:multiplicativedirac}

Let $G$ be a Lie groupoid over $M$. A Dirac structure $L_G$ on $G$ is said to be \textbf{multiplicative} if $L_G\subseteq TG\oplus T^{*}G$ is a subgroupoid over some subbundle $E\subseteq TM\oplus A^*G$. 

\end{definition}

We refer to a pair $(G,L_G)$, made up of a Lie groupoid $G$ and a multiplicative Dirac structure $L_G$ on $G$, as a \textbf{Dirac groupoid}. We use the notation $\mathrm{Dir}_{\textit{mult}}(G)$ to indicate the set consisting of all multiplicative Dirac structures on $G$.

It follows from the multiplicativity of $L_G$ that $E\subseteq TM\oplus A^*G$ is a vector subbundle. In particular, a multiplicative Dirac structure $L_G$ on a Lie groupoid $G$ defines a $\mathcal{VB}$-subgroupoid $L_G\subseteq \mathbb{T}G$.

\begin{example}

Let $\omega_G$ be a closed multiplicative $2$-form on a Lie groupoid $G$. The multiplicativity property of $\omega_G$ is equivalent to saying that the bundle map $\omega^{\sharp}_G:TG\rmap T^*G$ is a morphism of Lie groupoids. Hence, the corresponding Dirac structure $L_{\omega_G}=\text{Graph}(\omega_G)\subseteq \mathbb{T}G$ is a multiplicative Dirac structure. In this case we have a groupoid $L_{\omega_G}\rightrightarrows E$ where $E\subseteq TM\oplus A^*G$ is the subbundle given by the graph of the bundle map $-\sigma^t$ determined by the \textbf{IM-$2$-form} (see \cite{BCWZ}) $\sigma$ associated to $\omega_G$.

\end{example}

\begin{example}

Let $(G,\pi_G)$ be a Poisson groupoid. The multiplicativity of $\pi_G$ is equivalent to saying that $\pi^{\sharp}_G:T^*G\rmap TG$ is a morphism of Lie groupoids. Therefore, the associated Dirac structure $L_{\pi_G}=\text{Graph}(\pi_G)\subseteq \mathbb{T}G$ defines a multiplicative Dirac structure. In this case we have a groupoid $L_{\pi_G}\rightrightarrows E$ where $E\subseteq TM\oplus A^*G$ is the subbundle given by the graph of dual anchor map $\rho_{A^*G}:A^*G\rmap TM$

\end{example}

The examples discussed previously show that Dirac groupoids lead to a natural generalization of Poisson groupoids and presymplectic groupoids. Our main aim is to describe Dirac groupoids infinitesimally, establishing in particular, a connection between such a infinitesimal description and Lie bialgebroids and IM-$2$-forms.

Haw
\subsection{More examples of multiplicative Dirac structures}

In addition to multiplicative closed $2$-forms and multiplicative Poisson bivectors, there are several interesting multiplicative Dirac structures, which will be discussed throughout this subsection.

\subsubsection{\textbf{Foliated Groupoids}}

A regular distribution $F_G\subseteq TG$ is called \textbf{multiplicative} if it defines a Lie subgroupoid of the tangent groupoid $TG$. A \textbf{foliated groupoid} is a pair $(G,F_G)$ where $G$ is a Lie groupoid and $F_G$ is an involutive multiplicative regular distribution. In this case, the Dirac structure $F_G\oplus F^{\circ}_G\subseteq \mathbb{T}G$ is easily seen to be a multiplicative Dirac structure on $G$. The foliation tangent to an involutive multiplicative distribution is called a \textbf{multiplicative foliation}. Multiplicative foliations which are simultaneously transversal to the $s$-fibration and to the $t$-fibration were studied in \cite{T}, providing interesting examples of noncommutative Poisson algebras. Also, multiplicative foliations arise in the context of geometric quantization of symplectic groupoids, namely, as polarizations compatible with a symplectic groupoid structure (see \cite{Eli}). In addition, the notion of multiplicative foliation has appeared in connection with exterior differential systems. For more details see \cite{CSS} and the references therein.

\subsubsection{\textbf{Dirac Lie groups}}

Dirac Lie groups, that is, Lie groups equipped with multiplicative Dirac structures were first studied by the author in \cite{Ortiz1} providing a generalization of Poisson Lie groups within the category of Lie groups. In that work, it is shown that, modulo regularity issues, Dirac Lie groups are given by the pull back (in the sense of Dirac structures) of Poisson Lie groups via a surjective submersion which is also Lie group morphism. Notice that whenever a Lie groupoid $G$ over $M$ is equipped with a multiplicative Dirac structure, then for every $x\in M$, the isotropy Lie group $G_{x}:=s^{-1}(x)\cap t^{-1}(x)$ inherits a Dirac structure $L_{G_x}$ making the pair $(G_x,L_{G_x})$ into a Dirac Lie group. 

We emphasize that different notions of Dirac Lie groups exist in the literature. For instance, Li-Bland and Meinrenken have proposed in \cite{LiBland-Meinrenken} a notion of multiplicativity which includes interesting examples of \emph{twisted} Dirac structures on Lie groups such as the Cartan-Dirac structure on a compact Lie group.

\subsubsection{\textbf{Tangent lift of a multiplicative Dirac structure}}

It was proved in \cite{GrabUrbanski} that whenever a Lie group $G$ carries a multiplicative Poisson bivector $\pi_G$, then the tangent Lie group $TG$ equipped with the tangent Poisson structure $\pi_{TG}$ becomes a Poisson Lie group. It is easy to extend the multiplicative Poisson case to abstract multiplicative Dirac structures on Lie groupoids. Assume that $G$ is a Lie groupoid over $M$ and consider the tangent groupoid $TG$ over $TM$ explained in section \ref{tangentgroupoids}. Then, the tangent Dirac structure $L_{TG}\subseteq TTG\oplus T^*TG$ induced by a multiplicative Dirac structure $L_G\subseteq TG\oplus T^*G$ is also a multiplicative Dirac structure. Indeed, first observe that the bundle map $J_G:TTG\rmap TTG$ is a groupoid isomorphism over $J_M:TTM\rmap TTM$. Similarly, the bundle map $\Theta_G:TT^*G\rmap T^*TG$ is a groupoid isomorphism over the canonical identification $I:T(A^*G)\rmap (T(AG))^*$. Since $L_G$ is a Lie subgroupoid of $TG\oplus T^*G$, then the tangent functor yields a Lie subgroupoid $TL_G$ of $TTG\oplus TT^*G$. Due to the fact that $L_{TG}$ is the image of $TL_G$ via the groupoid isomorphism $J_G\oplus \Theta_G$, we see that $L_{TG}$ inherits a natural structure of Lie subgroupoid of $TTG\oplus T^*TG$. Hence we conclude that $L_{TG}$ defines a multiplicative Dirac structure on $TG$.

\subsubsection{\textbf{Symmetries of multiplicative Dirac structures}}\label{symmetriesmultiplicativedirac}

 Let $L_G$ be a multiplicative Dirac structure on a Lie groupoid $G\rightrightarrows M$, and let $H$ be a Lie group acting on $G$ by groupoid automorphisms. Assume that the $H$-action is free and proper and that the $H$-orbits coincide with the characteristic leaves of $L_G$. In this case the quotient space $G/H$ inherits the structure of a Lie groupoid over $M/H$. Moreover, since $G/H$ is the space of characteristic leaves of $L_G$, we conclude that there exists a Poisson structure $\pi_{red}$ on $G/H$, making the quotient map $G\rmap G/H$ into both a backward and forward Dirac map. This fact together with the multiplicativity of $L_G$ imply that $\pi_{red}$ is a multiplicative Poisson bivector. In other words, the quotient space $G/H$ is a Poisson groupoid. In the case where $L_G$ is the graph of a multiplicative Poisson bivector and the action is Hamiltonian in the sense of \cite{RuiDavid}, this recovers some of the results about reduction of Poisson groupoids carried out in \cite{RuiDavid}.

\subsubsection{\textbf{Multiplicative $B$-field transformations}}

Let $L\subseteq \mathbb{TM}$ be a Lagrangian subbundle. Given a $2$-form $B\in\Omega^2(M)$ one can construct the Lagrangian subbundle $\tau_{B}(L)\subseteq \mathbb{T}M$ defined by

$$\tau_B(L)=\{X\oplus\alpha + i_{X}B\mid X\oplus\alpha\in L \}.$$

A straightforward computation shows that $\tau_B(L)$ defines a Dirac structure on $M$ if and only if $B$ is a closed $2$-form. See for instance \cite{B, Gualtieri}. In this case, we say that the Dirac structure $\tau_B(L)$ is obtained out of $L$ by a \textbf{$B$-field transformation}. 

Assume now that $L_G$ is a multiplicative Dirac structure on a Lie groupoid $G$. Given a multiplicative closed $2$-form $B_G$ on $G$, one can consider the bundle map  $\tau_{B_G}:\mathbb{T}G\rmap \mathbb{T}G$, $X\oplus \alpha \mapsto X\oplus \alpha + i_X(B_G)$. It follows from the multiplicativity of $B_G$ that $\tau_{B_G}$ is a Lie groupoid isomorphism. As a result, the Dirac structure $\tau_{B_G}(L_G)$ on $G$ is multiplicative. Our interest on $B$-field transformations of multiplicative Dirac structures is motivated by the work carried out in \cite{B, BR}, where the authors study the connection between certain $B$-field transformations of symplectic and Poisson groupoids and the notion of Morita equivalence of Poisson manifolds.

\subsubsection{\textbf{Generalized Complex Groupoids}}

Generalized complex structures were introduced by Hitchin \cite{Hitchin} and further developed by Gualtieri \cite{Gualtieri}. Given a smooth manifold $M$, one can consider the complexified vector bundle $\mathbb{T}_{\C}M:=\mathbb{T}M\otimes \C$ endowed with the complex Courant bracket and the complex pairing $\langle\cdot,\cdot\rangle$. A generalized complex structure on $M$ is a complex Dirac structure $L\subseteq \mathbb{T}_{ \C}M$ such that $L\cap\overline{L}=\{0\}$, where $\overline{L}$ denotes the conjugate of $L$. Complexified versions of multiplicative Dirac structures gives rise to generalized complex groupoid. More concretely, let $G$ be a Lie groupoid equipped with a generalized complex structure $L_G$. We say that $(G,L_G)$ is a \textbf{generalized complex groupoid} if $L_G\subseteq \mathbb{T}_{\C}G$ is a Lie subgroupoid. Generalized complex groupoids were introduced in \cite{JotzStienonXu} under the name of Glanon groupoids. Structures such as symplectic groupoids and holomorphic Poisson groupoids are special instances of generalized complex groupoids.



\section{Dirac Algebroids}\label{DiracAlgebroids}

In this section we study Lie algebroids equipped with Dirac structures compatible with both the linear and Lie algebroid structure.

\subsection{Definition and main examples}

Let $A\rmap M$ be a vector bundle. A Poisson bivector $\pi_A$ on $A$ is \textbf{linear} if the map $\pi^{\sharp}_A:T^*A\rmap TA$ is a morphism of double vector bundles. Similarly, a $2$-form $\omega_A$ on $A$ is \textbf{linear} if the map $\omega^{\sharp}_A:TA\rmap T^*A$ is a morphism of double vector bundles. In this case, the bundle map $\omega^{\sharp}_A$ covers a bundle morphism $\lambda:TM\rmap A^*$. As shown in \cite{Kurbanski}, a linear $2$-form $\omega_A$ on a vector bundle $A\rmap M$ is closed if and only if $\omega_A=-(\lambda^t)^*\omega_{can}$, where $\omega_{can}$ is the canonical symplectic form on $T^*M$ and $\lambda^t:A\rmap T^*M$ is a fiberwise dual map of $\lambda:TM\rmap A^*$. The definition below includes both linear Poisson bivectors and linear closed $2$-forms as special instances.

\begin{definition}

A Dirac structure $L_A$ on $A$ is called \textbf{linear} if $L_A\subseteq \mathbb{T}A$ is a double vector subbundle of $\mathbb{T}A$.

\end{definition}

A linear Dirac structure $L_A\subseteq \mathbb{T}A$ is not only a vector bundle over $A$, but also a vector bundle over a subbundle $E\subseteq TM\oplus A^*$. It follows directly from the definition that graphs of linear Poisson bivector and linear closed $2$-forms define linear Dirac structures. Linear Dirac structures arise also in connection with Lagrangian and Hamiltonian mechanics, see e.g. \cite{GrabGrab}.

Assume now that $A\rmap M$ carries also a Lie algebroid structure. Consider the direct sum Lie algebroid $\mathbb{T}A=TA\oplus T^*A$, whose base manifold is $TM\oplus A^*$.

\begin{definition}

A Dirac structure $L_A$ on $A$ is called \textbf{morphic} if $L_A$ is a linear Dirac structure which is also a Lie subalgebroid of $\mathbb{T}A$.

\end{definition}

We denote by $\mathrm{Dir}_{morph}(A)$ the space of morphic Dirac structures on the Lie algebroid $A$.

A pair $(A,L_A)$ where $A$ is a Lie algebroid endowed with a morphic Dirac structure $L_A$ will be referred to as a \textbf{Dirac algebroid}.

\begin{example}

Let $\pi_A$ be a linear Poisson bivector on a Lie algebroid $A\rmap M$. Then, the Dirac structure given by the graph of $\pi_A$ is morphic if and only if $\pi^{\sharp}_A:T^*A\rmap TA$ is a Lie algebroid morphism. As shown in \cite{MX1}, this is equivalent to the pair $(A,A^*)$ being a Lie bialgebroid.

\end{example}

\begin{example}

Let $\omega_A$ be a linear closed $2$-form on a Lie algebroid $A\rmap M$, i.e. $\omega_A=-\sigma^*\omega_{can}$, for some bundle map $\sigma:A\rmap T^*M$. The Dirac structure defined by the graph of $\omega_A$ is morphic if and only if $\omega^{\sharp}_A:TA\rmap T^*A$ is a Lie algebroid morphism. Equivalently, as shown in \cite{BCO}, the bundle map $\sigma:A\rmap T^*M$ is an IM-$2$-form on $A$. The notion of IM-$2$-form was introduced in \cite{BCWZ} motivated by the problem of the integration of Dirac structures. See also \cite{AriasCrainic} where IM-$2$-forms arise in connection with the Weil algebra and the Van Est isomorphism.

\end{example}

\subsection{More examples of Dirac algebroids}

In addition to both morphic Poisson structures and morphic closed $2$-forms, there are more examples of morphic Dirac structures, which we proceed to explain them below.

\subsubsection{\textbf{Foliated algebroids}}

Let $A$ be a Lie algebroid and $F_A\subseteq TA$ an involutive subbundle which is also a Lie subalgebroid of $TA\rmap TM$. In this case we say that $(A,F_A)$ is a \textbf{foliated algebroid}. One can easily check that the Dirac structure $F_A\oplus F^{\circ}_A\subseteq \mathbb{T}A$ is a morphic Dirac structure. Foliated algebroids were studied in \cite{Eli} as a way to promote the notion of polarization in geometric quantization to the category of Lie algebroids. Also, a detailed discussion about foliated algebroids can be found in \cite{JotzOrtiz}.  

\subsubsection{\textbf{Dirac Lie algebras}}

Let $\mathfrak{g}$ be a Lie algebra. In this case, morphic Dirac structures are Lie subalgebroids of $T\mathfrak{g}\oplus T^*\mathfrak{g}\rmap \mathfrak{g}^*$. It follows from \cite{Ortiz1} that Dirac Lie algebras are suitable pull backs of Lie bialgebras.

\subsubsection{\textbf{Tangent lifts of Dirac algebroids}}\label{tangentliftmorphic}

Let $(A,L_A)$ be a Dirac algebroid. Consider the tangent Dirac structure $L_{TA}$ on $TA$. By definition, the tangent Dirac structure is given by $L_{TA}:= (J_A\oplus \Theta_A)(TL_A)$, where $TL_A\rmap TM$ is the tangent algebroid associated to the Dirac structure $L_A$ viewed as a Lie algebroid over $A$. Since the bundle map $J_A\oplus \Theta_A:TTA\oplus TT^*A\rmap TTA\oplus T^*TA$ is a Lie algebroid isomorphism, we conclude that $L_{TA}\subseteq \mathbb{T}TA$ is a Lie subalgebroid. Therefore, the pair $(TA,L_{TA})$ is a Dirac algebroid.

\subsubsection{\textbf{Symmetries of Dirac algebroids}}

Let $(A,L_A)$ be a Dirac algebroid. Consider a Lie group $H$ acting on $A$ by Lie algebroid automorphisms. Assume that the action is free and proper and that the $H$-orbits coincide with the characteristic leaves of $L_A$. One can check that the orbit space $A/H$ inherits a Lie algebroid structure over $M/H$, making the quotient map $A\rmap A/H$ into a Lie algebroid morphism. Since the $H$-orbits are exactly the characteristic leaves of $L_A$, one concludes that $A/H$ is equipped with a unique Poisson bivector $\pi_{red}$ determined by the fact that $A\rmap A/H$ is a forward and backward Dirac map. Since $L_A$ is morphic, we conclude that $\pi_{red}$ is a morphic Poisson structure on $A/H$. In particular, due to \cite{MX1}, the pair $(A/H, (A/H)^*)$ is a Lie bialgebroid.  In the special case where $L_A$ is the graph of a morphic Poisson structure on $A$ and the action is Hamiltonian in the sense of \cite{RuiDavid}, this recovers the reduction of Lie bialgebroids carried out in \cite{RuiDavid}.

\subsubsection{\textbf{Morphic $B$-field transformations}}

Let $(A,L_A)$ be a Dirac algebroid. Associated to a morphic closed $2$-form $B_A$ on $A$ is the Lie algebroid automorphism $\tau_{B_A}:\mathbb{T}A\rmap \mathbb{T}A, (X,\alpha)\mapsto (X,\alpha + i_XB_A)$. The Dirac structure $\tau_{B_A}(L_A)\subseteq \mathbb{T}A$ obtained out of $L_A$ by applying the B-field transformation $\tau_{B_A}$ is morphic. Therefore, the pair $(A,\tau_{B_A})$ is a Dirac algebroid. In particular, $B$-field transformations of morphic Poisson structures (i.e. Lie bialgebroid structures on $(A,A^*)$) by morphic closed $2$-forms are always morphic Dirac structures. If the $B$-field transformation is admissible, that is, the resulting Dirac structure is the graph of a Poisson bivector, such a bivector is necessarily morphic as well. In particular, we get a new bialgebroid structure on $(A,A^*)$ referred to as a \textbf{gauge transformation} of the Lie bialgebroid $(A,A^*)$. Gauge transformations of Lie bialgebroids were introduced in \cite{B} motivated by the study of gauge transformations of Poisson groupoids and Morita equivalence of Poisson manifolds.

\subsubsection{\textbf{Generalized Complex Algebroids}}

Let $A\rmap M$ be a Lie algebroid. Consider the complexified Lie algebroid $\mathbb{T}_{\C}A=(TA\oplus T^*A)\otimes \C$ whose base manifold is $(TM\oplus A^*)\otimes \C$. A generalized complex structure $L_A$ on $A$ is \textbf{morphic} if $L_A\subseteq \mathbb{T}_{\C}A$ is a Lie subalgebroid. In this case, we say that $(A,L_A)$ is a \textbf{generalized complex algebroid}. The notion of generalized complex algebroid was introduced in \cite{JotzStienonXu} under the name of Glanon algebroids. Generalized complex algebroids include holomorphic Poisson structures and holomorphic Lie bialgebroids as particular cases.


\section{Infinitesimal description of multiplicative Dirac structures}\label{frommulttolin}

This section is the main part of the present work. Here we show that Dirac algebroids correspond to the infinitesimal counterpart of Dirac groupoids.

\subsection{The canonical $\mathcal{CA}$-groupoid}

The main idea for studying multiplicative Dirac structures infinitesimally, is based on the following observation. Given a Lie groupoid $G$ over $M$, the canonical geometric objects associated to $\mathbb{T}G$ that are used to define Dirac structures (symmetric pairing and Courant bracket) are suitably compatible with the groupoid structure of $\mathbb{T}G$. This compatibility makes $\mathbb{T}G$ into a $\mathcal{CA}$-groupoid. The notion of $\mathcal{CA}$-groupoid was suggested by Mehta in \cite{Mehta} and further studied by Li-Bland and \v{S}evera \cite{LiblandSevera}. More precisely, let  $\langle\cdot,\cdot\rangle_G$ be the nondegenerate symmetric pairing on the direct sum Lie groupoid $\mathbb{T}G$.

\begin{proposition}\label{multiplicativepairing}
The canonical pairing defines a morphism of Lie groupoids 
$$\langle\cdot,\cdot\rangle_G:\mathbb{T}G\oplus \mathbb{T}G\rmap \R,$$

\noindent where $\R$ is equipped with the usual abelian group structure. 

\end{proposition}

\begin{proof}
 
Since $\R$ is a groupoid over a point, we only need to check the compatibility of $\langle\cdot,\cdot\rangle_G$ with the corresponding groupoid multiplications. For that, consider elements $(X_g\oplus\alpha_g),(Y_g\oplus\beta_g)\in \mathbb{T}_gG$ and $(X'_h\oplus\alpha'_h),(Y'_h\oplus\beta'_h)\in \mathbb{T}_hG$. Then by definition of the groupoid structure on $\mathbb{T}G\oplus\mathbb{T}G$, we have

$$((X_g\oplus\alpha_g)\oplus(Y_g\oplus\beta_g))*((X'_h\oplus\alpha'_h)\oplus(Y'_h\oplus\beta'_h))=(X_g\bullet X'_h \oplus \alpha_g\circ\alpha'_h)\oplus(Y_g\bullet Y'_h\oplus\beta_g\circ\beta'_h),$$

\noindent therefore one gets

\begin{align*}
\langle (X_g\bullet X'_h\oplus\alpha_g\circ\alpha'_h),(Y_g\bullet Y'_h\oplus\beta_g\circ\beta'_h)\rangle_G
=&(\alpha_g\circ\alpha'_h)(Y_g\bullet Y'_h) + (\beta_g\circ\beta'_h)(X_g\bullet X'_h)\\
=& \alpha_g(Y_g)+\alpha'_h(Y'_h) + \beta_g(X_g) + \beta'_h(X'_h)\\
=& \langle (X_g\oplus\alpha_g),(Y_g,\beta_g)\rangle_G + \langle (X'_h\oplus\alpha'_h),(Y'_h\oplus\beta'_h)\rangle_G
\end{align*}

This proves the statement.
\end{proof}

In order to explain the relation between the Courant bracket and the Lie groupoid structure on the direct sum vector bundle $\mathbb{T}G=TG\oplus T^*G$, we consider the direct product Courant algebroid $\mathbb{T}G\times\mathbb{T}G\rmap G\times G$. Every section $a^{(2)}$ of $\mathbb{T}G\times \mathbb{T}G$ can be written as 

$$a^{(2)}=a_1\circ pr_1\oplus a_2\circ pr_2,$$

 \noindent where $a_1,a_2$ are sections of $\mathbb{T}G$, and $pr_1,pr_2:\mathbb{T}G\times\mathbb{T}G\rmap \mathbb{T}G$ denote the natural projections. The direct product bracket on sections of $\mathbb{T}G\times\mathbb{T}G$ is defined as usual

$$[a^{(2)},\overline{a}^{(2)}]=\Cour{a_1,\overline{a}_1}\circ pr_1\oplus \Cour{a_2,\overline{a}_2}\circ pr_2,$$

\noindent and the anchor map $\rho_{(\mathbb{T}G)_{(2)}}:\mathbb{T}G\times \mathbb{T}G\rmap TG\times TG$ is given by the canonical componentwise projection.

\begin{proposition}\label{multiplicationLG}

Let $m_{\mathbb{T}}:(\mathbb{T}G)_{(2)}\rmap \mathbb{T}G$ denote the groupoid multiplication of $\mathbb{T}G=TG\oplus T^*G$. If $a,b,a_i,b_i\in \Gamma(\mathbb{T}G), i=1,2$ are sections such that 

$$m_{\mathbb{T}}\circ (a_1,a_2)=a\circ m_G; \quad m_{\mathbb{T}}\circ(b_1,b_2)=b\circ m_G,$$

\noindent then the following identities hold

\begin{itemize}

\item[i)] $Tm_G(\rho_{(\mathbb{T}G)_{(2)}}(X^1_g\oplus \alpha^1_g, X^2_h\oplus \alpha^2_h))=X^1_g\bullet X^2_h;$

\item[ii)] $m_{\mathbb{T}}\circ (\Cour{a_1,b_1},\Cour{a_2,b_2})=\Cour{a,b}\circ m_G.$

\end{itemize}

\end{proposition}

\begin{proof}
 
We begin by checking the identity i). For that, consider a section $a^{(2)}=a_1\circ pr_1 \oplus a_2\circ pr_2$ of $(\mathbb{T}G)_{(2)}$ where $a_1=X^1\oplus \alpha^1$ and $a_2=X^2\oplus \alpha^2$ are sections of $\mathbb{T}G$. The multiplication on the Lie groupoid $\mathbb{T}G$ maps the section $a^{(2)}$ into

$$m_{\mathbb{T}}(a_1\circ pr_1\oplus a_2\circ pr_2)(g,h)=X^1_g\bullet X^2_h \oplus \alpha^1_g\circ \alpha^2_h.$$

\noindent Applying the anchor map of $\mathbb{T}G$ we obtain

$$\rho_{\mathbb{T}G}(X^1_g\bullet X^2_h \oplus \alpha^1_g\circ \alpha^2_h)=X^1_g\bullet X^2_h.$$

\noindent On the other hand, the componentwise anchor map of $(\mathbb{T}G)_{(2)}$ applied to the section $a^{(2)}$ gives rise to

$$\rho_{(\mathbb{T}G)_{(2)}}(a_1\circ pr_1\oplus a_2\circ pr_2)(g,h)=(X^1_g,X^2_h),$$

\noindent which followed by the derivative of $m_G:G_{(2)}\rmap G$ yields

$$Tm_G(\rho_{(\mathbb{T}G)_{(2)}}(X^1_g\oplus \alpha^1_g, X^2_h\oplus \alpha^2_h))=X^1_g\bullet X^2_h,$$

\noindent as required. In order to prove identitiy ii), one considers

\begin{align}\label{mrelated1}
m_{\mathbb{T}}\circ a^{(2)}=&a\circ m_G\\
m_{\mathbb{T}}\circ \overline{a}^{(2)}=&\overline{a}\circ m_G\label{mrelated2},
\end{align}

\noindent where $a^{(2)},\overline{a}^{(2)}\in \Gamma_{G_{(2)}}((\mathbb{T}G)_{(2)})$ and $a,\overline{a}\in \Gamma_{G}(\mathbb{T}G)$. More concretely, write down sections as

\begin{align*}
 a^{(2)}=&(X^1\oplus \alpha^1)\circ pr_1 \oplus (X^2\oplus \alpha^2)\circ pr_2\\ 
 \overline{a}^{(2)}=&(\overline{X}^1\oplus \overline{\alpha}^1)\circ pr_1 \oplus (\overline{X}^2\oplus \overline{\alpha}^2)\circ pr_2\\
a=& Y\oplus \beta\\
\overline{a}=&\overline{Y}\oplus \overline{\beta},
\end{align*}

\noindent then the identities (\ref{mrelated1}), (\ref{mrelated2}) become

\begin{align}\label{mrelated1prima}
 X^1_{g}\bullet X^2_{h}\oplus \alpha^1_{g}\circ \alpha^2_{h}=& Y_{gh}\oplus \beta_{gh}\\
 \overline{X}^1_{g}\bullet \overline{X}^2_{h}\oplus \overline{\alpha}^1_{g}\circ \overline{\alpha}^2_{h}=& \overline{Y}_{gh}\oplus \overline{\beta}_{gh},\label{mrelated2prima}
\end{align}

\noindent for any composable pair $(g,h)\in G\times G$. Now it follows directly from the definition of the direct product bracket that

$$[a^{(2)},\overline{a}^{(2)}]=([X^1,\overline{X}^1]\oplus \Lie_{X^1}\overline{\alpha}^1-i_{\overline{X}^1}d\alpha^1)\circ pr_1\oplus ([X^2,\overline{X}^2]\oplus \Lie_{X^2}\overline{\alpha}^2-i_{\overline{X}^2}d\alpha^2)\circ pr_2.$$ 

\noindent Then, composing with the groupoid multiplication of $\mathbb{T}G$, we have

\begin{align*}
 m_{\mathbb{T}}\circ [a^{(2)},\overline{a}^{(2)}]_{(g,h)}=[X^1,\overline{X}^1]_{g}\bullet[X^2,\overline{X}^2]_{h}\oplus (\Lie_{X^1}\overline{\alpha}^1-i_{\overline{X}^1}d\alpha^1)_{g}\circ(\Lie_{X^2}\overline{\alpha}^2-i_{\overline{X}^2}d\alpha^2)_{h}.
\end{align*}

\noindent On the other hand, 

$$\Cour{a,\overline{a}}\circ m_G(g,h)=[Y,\overline{Y}]_{gh}\oplus (\Lie_{Y}\overline{\beta}-i_{\overline{Y}}d\beta)_{gh},$$

\noindent and using the identities (\ref{mrelated1prima}) and (\ref{mrelated2prima}) one concludes that

$$[Y,\overline{Y}]_{gh}=[X^1,\overline{X}^1]_{g}\bullet [X^2,\overline{X}^2]_{h}.$$

\noindent Thus, the tangent component of $\Cour{a,\overline{a}}_{gh}$ coincides with the tangent component of $m_{\mathbb{T}}\circ[a^{(2)},\overline{a}^{(2)}]_{(g,h)}$. It remains to show that we also have the equality of the corresponding cotangent parts. This is equivalent to showing that

\begin{align*}
 (\Lie_{Y}\overline{\beta}-\Lie_{\overline{Y}}\beta - d\langle \beta,\overline{Y}\rangle)_{gh}=&(\Lie_{X^1}\overline{\alpha}^1-\Lie_{\overline{X}^1}\alpha^1-d\langle \alpha^1,\overline{X}^1\rangle)_{g}\circ\\ \circ&(\Lie_{X^2}\overline{\alpha}^2-\Lie_{\overline{X}^2}\alpha^2-d\langle \alpha^2,\overline{X}^2\rangle)_{h},
\end{align*}

\noindent for every composable pair $(g,h)\in G_{(2)}$. In order to prove this identity, we need to check that the left hand side ($LHS$), and the right hand side ($RHS$) above coincide at elements of the form $U_g\bullet V_h$. For that consider the $1$-form on $G$ defined by $\gamma:=\Lie_{Y}\overline{\beta}-\Lie_{\overline{Y}}\beta-d\langle\beta,\overline{Y}\rangle.$ We can look at the pull back $1$-form $m^*_{G}\gamma\in\Omega^{1}(G_{(2)})$, which at every tangent vector $(U_g,V_h)\in T_{(g,h)}G_{(2)}$ is given by

$$(m^*_{G}\gamma)_{(g,h)}(U_g,V_h)=\gamma_{gh}(U_g\bullet V_h)=(LHS)(U_g\bullet V_h).$$

The pull back form $m^*_{G}\gamma$ involves three terms. Let us analyze the first term $m^*_{G}(\Lie_{Y}\overline{\beta})$ of this pull back form. It follows from the relation $Y=(m_G)_{*}(X^1,X^2)$ that

$$m^*_{G}(\Lie_{Y}\overline{\beta})=\Lie_{(X^1,X^2)}m^*_{G}\overline{\beta}.$$

\noindent Notice that (\ref{mrelated2prima}) implies that

\begin{align*}
 (m^*_{G}\overline{\beta})_{(g,h)}(U_g,V_H)=&\overline{\beta}_{gh}(U_g\bullet V_h)\\
=&(\overline{\alpha}^1_{g}\circ \overline{\alpha}^2_{h})(U_g\bullet V_h)\\
=&\overline{\alpha}^1_{g}(U_g) + \overline{\alpha}^2_{h}(V_h)\\
=&(\overline{\alpha}^1,\overline{\alpha}^2)_{(g,h)}(U_g,V_h).
\end{align*}

\noindent That is, $m^*_{G}(\Lie_{Y}\overline{\beta})=\Lie_{X^1}\overline{\alpha}^1\oplus \Lie_{X^2}\overline{\alpha}^2$. A similar argument can be applied to the other terms of the pull back form $m^*_{G}\gamma$, yielding

\begin{align*}
(LHS)(U_g\bullet V_h)=&(m^*_{G}\gamma)_{(g,h)}(U_g,V_h)\\
=&(\Lie_{X^1}\overline{\alpha}^1)_{g}(U_g) + (\Lie_{X^2}\overline{\alpha}^2)_{h}(V_h)+\\
-&(\Lie_{\overline{X}^1}\alpha^1)_{g}(U_g) - (\Lie_{\overline{X}^2}\alpha^2)_{h}(V_h)+\\
-& d\langle\alpha^1,\overline{X}^1\rangle_g(U_g) - d\langle\alpha^2,\overline{X}^2\rangle_h(V_h)\\ 
=&(RHS)(U_g\bullet V_h).
\end{align*}

Thus $RHS$ and $LHS$ coincide at elements of the form $U_g\bullet V_h$, and we conclude that $(m_{\mathbb{T}},m_G)$ is bracket preserving.

\end{proof}

Recall that, given a Courant algebroid $(\mathbb{E},\rho,\Cour{\cdot,\cdot})$ over smooth manifold $M$ and a submanifold $Q\subseteq M$, a \textbf{Dirac structure supported} on $Q$ (see \cite{AX,BIS}) is a subbundle $K\subset \mathbb{E}|_{Q}$ such that $K_x\subseteq \mathbb{E}_x$ is Lagrangian for all $x\in Q$ and the following conditions are fulfilled:

\begin{enumerate}

\item $\rho(K)\subseteq TQ$;

\item whenever $a_1,a_2\in\Gamma(\mathbb{E})$ satisfy $a_1|_{Q},a_2|_{Q}\in \Gamma(K)$, then $\Cour{a_1,a_2}|_{Q}\in\Gamma(K)$.

\end{enumerate}

Dirac structures with support were used in \cite{BIS} to introduce a natural notion of morphism between Courant algebroids. Let $\mathbb{E}_1,\mathbb{E}_2$ Courant algebroids over $M,N$, respectively. A \textbf{Courant algebroid morphism} from $\mathbb{E}_1$ to $\mathbb{E}_2$  is a Dirac structure in $\mathbb{E}_2\times \overline{\mathbb{E}_1}$ supported on $\mathrm{graph}(f)$ where $f:M\rmap N$ is a smooth map. Here $\mathbb{E}_1$ denotes the Courant algebroid structure on the vector bundle $\mathbb{E}_1$ with the same bracket on $\Gamma(\mathbb{E}_1)$, anchor map and minus the usual symmetric pairing.

Combing Proposition \ref{multiplicativepairing} and Proposition \ref{multiplicationLG}, we obtain the following.

\begin{proposition}\label{canonicalCAgroupoid}

Let $G$ be a Lie groupoid over $M$ with multiplication map $m_G:G_{(2)}\rmap G$. Let $m_{\mathbb{T}}:(\mathbb{T}G)_{(2)}\rmap \mathbb{T}G$ denote the groupoid multiplication on $\mathbb{T}G$. Then $\mathrm{graph}(m_{\mathbb{T}})\subseteq \mathbb{T}G\times \overline{\mathbb{T}G\times \mathbb{T}G}$ is a Dirac structure supported on $\mathrm{graph}(m_G)\subseteq G\times G\times G$. That is, $\mathrm{graph}(m_{\mathbb{T}})$ is a Courant algebroid morphism from $\mathbb{T}G\times \mathbb{T}G$ to $\mathbb{T}G.$

\end{proposition}

Using the terminology of \cite{LiblandSevera}, Proposition \ref{canonicalCAgroupoid} says that $\mathbb{T}G$ with its canonical Courant algebroid structure and groupoid multiplication is an example of $\mathcal{CA}$-groupoid.


\subsection{The $\mathcal{LA}$-groupoid of a multiplicative Dirac structure}

\subsubsection{Review of $\mathcal{LA}$-groupoids}

An $\mathcal{LA}$-groupoid is a Lie groupoid object in the category of Lie algebroids. More precisely, an \textbf{$\mathcal{LA}$-groupoid} \cite{MacdoubleI} is a square 

\begin{align}\label{LAgroupoid}
\xy 
(-15,10)*+{H}="t0"; (-15,-10)*+{E}="b0"; 
(15,10)*+{G}="t1"; (15,-10)*+{M}="b1"; 
{\ar@<.25ex>^{q_{H}} "t0"; "t1"}; 
{\ar@<.25ex>_{q_E} "b0"; "b1"}; 
{\ar@<.5ex> "t0"; "b0"}; 
{\ar@<-.5ex> "t0"; "b0"}; 
{\ar@<.5ex> "t1"; "b1"}; 
{\ar@<-.5ex> "t1"; "b1"}; 
\endxy
\end{align}

\noindent where the single arrows denote Lie algebroids and the double arrows denote Lie groupoids. These structures are compatible in the sense that all the structure mappings (i.e. source, target, unit section, inversion and multiplication) defining the Lie groupoid $H$ are Lie algebroid morphisms over the corresponding structure mappings which define the Lie groupoid $G$. We also require that the anchor map $\rho_{H}:H\rmap TG$ be a groupoid morphism over the anchor map $\rho_E:E\rmap TM$. Here $TG$ is endowed with the tangent groupoid structure over $TM$. For describing the square given by an $\mathcal{LA}$-groupoid we use the notation $(H,G,E,M)$. It is worthwhile to explain how the groupoid multiplication defines a morphism of Lie algebroids. For that, let $m_{H}: H_{(2)}\subseteq H\times H\rmap H$ denote the groupoid multiplication of $H$, and similarly let $m_G:G_{(2)}\subseteq G\times G\rmap G$ denote the multiplication of $G$. The direct product vector bundle $H\times H\rmap G\times G$ inherits a natural Lie algebroid structure, and we have a Lie subalgebroid $H_{(2)}$ over $G_{(2)}$ which is just a pull back algebroid, see e.g. \cite{HM} for details about the pull back operation in the category of Lie algebroids. With respect to this Lie algebroid structure, the multiplication map $m_{H}$ is required to be a Lie algebroid morphism covering $m_G$.

 


The Lie functor applied to an $\mathcal{LA}$-groupoid (\ref{LAgroupoid}) determines a double vector bundle

\begin{align}\label{doubleliealgebroid}
\xy 
(-15,10)*+{AH}="t0"; (-15,-10)*+{E}="b0"; 
(15,10)*+{AG}="t1"; (15,-10)*+{M}="b1"; 
{\ar@<.25ex>^{A(q_{H})} "t0"; "t1"}; 
{\ar@<.25ex>_{q_E} "b0"; "b1"}; 
{\ar@<.25ex> "t0"; "b0"}; 
{\ar@<.25ex> "t1"; "b1"}; 
\endxy
\end{align}

\noindent where each of the arrows define Lie algebroids. The top Lie algebroid structure is non trivial, and it deserves a detailed explanation. The Lie algebroid structure $AH\rmap AG$ was constructed in \cite{MacdoubleII} as a prolongation procedure similar to the tangent prolongation of a Lie algebroid, except that we replace the tangent functor by the Lie functor.

\begin{definition}

 The \textbf{prolonged anchor map} $AH\rmap T(AG)$ is defined by

$$\tilde{\rho}:=j^{-1}_{G}\circ A(\rho_{H}),$$

\noindent where $j_{G}:T(AG)\rmap A(TG)$ is the canonical identification defined in appendix A.

\end{definition}

Now we study the space of sections $\Gamma_{AG}(AH)$.

\begin{definition}

A section $u\in\Gamma_{G}(H)$ is called a \textbf{star section} if there exists a section $u_0\in\Gamma_{M}(E)$ such that

\begin{enumerate}
 \item $\epsilon_{E}\circ u_0=u\circ \epsilon_{M},$

 \item $s_{H}\circ u=u_0\circ s_{G}.$

\end{enumerate}

\end{definition}

Notice that since every star section $u:G\rmap H$ preserves the units and the source fibrations, we are allowed to apply the Lie functor to $u$, yielding a section $A(u)$ of the vector bundle $AH\stackrel{A(q_{H})}{\longrightarrow} AG$.

\begin{definition}
 
Let $(H,G,E,M)$ be an $\mathcal{LA}$-groupoid. The \textbf{core} of $H$ is the vector bundle over $M$ defined by

$$K:=\epsilon^{*}_{M}\Ker(s_{H}).$$

\end{definition}
 



Every section $k\in\Gamma(K)$ induces a section $k_{H}\in\Gamma_{G}(H)$ in the following way

$$k_{H}(g):= k(t_{G}(g))0^{H}_g,$$

\noindent where $0^{H}_g$ is the zero element in the fiber $H_g$ above $g\in G$. Notice that for every section $k\in\Gamma(K)$ the induced section $k_{H}\in\Gamma_{G}(H)$ satisfies 

$$k_{H}\circ\epsilon_M=k.$$






It was proved in \cite{MacdoubleII} that the core of the double vector bundle $(AH,AG,E,M)$ is the vector bundle $K\rmap M$. Notice that a core element $k\in K$ induces a Lie algebroid element $\overline{k}\in AH$. Indeed, we observe that every element in $AH$ has the form

$$W=\frac{d}{dt}(h_t)\vert_{t=0},$$ 

\noindent where $h_t$ is a curve in $H$ sitting in a fixed source fiber $s^{-1}_{H}(e)$ with $h_0=\epsilon_E(e)$. Thus, for every core element $k\in K$ above $x\in M$, that is $s_{H}(k)=0^{E}_{x}$ and $q_{H}(k)=\epsilon_M(x)$, there exists a natural element $\overline{k}\in AH$, defined by

$$\overline{k}:=\frac{d}{dt}(tk)\vert_{t=0}.$$

\begin{definition}
Given a section $k\in\Gamma(K)$, the \textbf{core} section induced by $k$ is the section $k^{\rm{core}}\in\Gamma_{AG}(AH)$ defined by

$$k^{\rm{core}}(u_x):=A(0^{H})u_x+ \overline{k(x)}.$$ 

\end{definition}

The following proposition was proved in \cite{MacdoubleII}.

\begin{proposition}
 
The space of sections $\Gamma_{AG}(AH)$ is generated by sections of the form $A(u)$, where $u:G\rmap H$ is a star section, and by sections of the form $k^{\rm{core}}$, where $k:M\rmap K$ is a section of the core of $H$. 

\end{proposition}

The Lie bracket on $\Gamma_{AG}(AH)$ is defined in terms of star sections and core sections. First we observe that whenever $u,v\in\Gamma_G(H)$ are star sections, then the Lie bracket $[u,v]\in\Gamma_{G}(H)$ is also a star section. Thus the Lie bracket between sections of the form $A(u),A(v)$ is defined by

$$[A(u),A(v)]=A([u,v]).$$

The bracket of a pair of core sections is defined by 

$$[k^{\rm{core}}_1,k^{\rm{core}}_2]=0.$$

In order to define the bracket of a star section and a core section we notice that every star section $u:G\rmap H$ induces a covariant differential operator 
\begin{align*}
D_u:\Gamma(K&)\rmap \Gamma(K)\\
&k\mapsto [u,k_{H}]\circ \epsilon_M,
\end{align*}

\noindent now we define $[A(u),k^{\rm{core}}]=(D_u(k))^{\rm{core}}.$

The Lie bracket of other sections of $\Gamma_{AG}(AH)$ is defined by requiring the Leibniz rule

$$[w,fw']=f[w,w']+(\Lie_{\tilde{\rho}(w)}f)w'.$$

The vector bundle $AH\stackrel{A(q_{H})}{\longrightarrow}AG$ endowed with the anchor map $\tilde{\rho}=j^{-1}_G\circ A(\rho)$ and the Lie bracket $[\cdot,\cdot]$ on $\Gamma_{AG}(AH)$ becomes a Lie algebroid called the \textbf{prolonged Lie algebroid} induced by $H\rmap G$, see \cite{MacdoubleII}.

Although the following remark is not mentioned in \cite{MacdoubleII}, it is important to notice that Mackenzie's construction of the prolonged Lie algebroid is natural in the following sense.

\begin{proposition}\label{natural}

Let $(H,G,E,M)$ be an $\mathcal{LA}$-groupoid. Consider the canonical embeddings $i_{AH}:AH\rmap TH$ and $i_{AG}: AG\rmap TG$. Endow $TH\rmap TG$ with the tangent algebroid structure and $AG\rmap AH$ with the prolonged algebroid structure. Then $i_{AH}$ is a Lie algebroid morphism covering $i_{AG}$. 

\end{proposition}

Recall that (see e.g. \cite{M}) a vector bundle map $\Psi:A\rmap B$, covering $\psi:M\rmap N$, is a \textbf{Lie algebroid morphism} if  $$\rho_B \circ \Psi = T\psi \circ \rho_A,$$ and the following compatibility with brackets holds: for sections $u,v\in \Gamma(A)$ such that $\Psi(u)=\sum_{j}f_j\psi^*u_j$ and $\Psi(v)=\sum_{i}g_i\psi^*v_i$, where $f_j,g_i \in C^\infty(M)$ and $u_j,v_i \in \Gamma(B)$, we have

\begin{equation}\label{brackets}
\Psi([u,v]_{\st{A}})=\sum_{i,j}f_jg_i\psi^*[u_j,v_i]_{\st{B}} +
\sum_{i}\Lie_{\rho_A(u)}g_i \psi^*v_i - \sum_{j}\Lie_{\rho_A(v)}f_j \psi^*u_j.
\end{equation}

\begin{proof}

The compatibility with the anchor maps reads

$$\rho_{TH}\circ i_{AH}=Ti_{AG}\circ \tilde{\rho},$$

\noindent which is exactly the definition of the prolonged anchor map.

Let us check now the compatibility with the Lie brackets. For that, consider a star section $u:G\rmap H$. Then, there are sections $Tu:TG\rmap TH$ and $A(u):AG\rmap AH$. Both are related by $A(u)=Tu|_{AG}$. In particular, the following identity holds $i_{AH}\circ A(u)=Tu\circ i_{AG}$. Similarly, every section $k\in\Gamma(K)$ of the core of $H$, induces a section of the tangent prolongation $TH\rmap TG$. Indeed, first consider the induced section $k_{H}\in\Gamma_{G}(H)$ and then construct the core section $\hat{k}_{H}\in\Gamma_{TG}(TH)$ defined in the usual way

$$\hat{k}_{H}(X_g)=T(0^{H})X_g+ \overline{k_{H}(g)}.$$

\noindent For every $x\in\epsilon_M(M)\subseteq G$ one has $k_{H}(x)=k(x)$, and thus at any $u_x\in (AG)_{x}\subseteq T_xG$ we get

$$\hat{k}_{H}(u_x)= A(0^{H})u_x+ \overline{k(x)}.$$

\noindent Hence we conclude that $i_{AH}\circ k^{core}=\hat{k}_H\circ i_{AG}$.  Let us show that equation \eqref{brackets} for a pair o sections $A(u), A(v)$, where $u,v:G\rmap H$ are star sections. Indeed,

$$i_{AH}\circ [A(u),A(v)]=i_{AH}\circ A[u,v]=T[u,v]\circ i_{AG}=[Tu,Tv]\circ i_{AG},$$

\noindent as desired. It remains to show the bracket condition \eqref{brackets} for sections of the form $A(u), k^{core}$, where $u:G\rmap H$ is a star section and $k:M\rmap K$ is a section of the core. On one hand, one has that

$$i_{AH}\circ [A(u),k^{core}]=i_{AH}\circ (D_{u}k)^{core}=(\widehat{D_{u}k})_H\circ i_{AG}$$

\noindent On the other hand, 

$$[Tu,\hat{k}_H]\circ i_{AG}=\widehat{[u,k_H]}\circ i_{AG}$$

Notice that, to conclude that \eqref{brackets} holds in this case, it suffices to show that $(\widehat{D_{u}k})_H\circ i_{AG}=\widehat{[u,k_H]}\circ i_{AG}$. Indeed, using the fact that $k=k_H\circ \epsilon_M$ for every section $k:M\rmap K$, we conclude that if $v_x\in A_xG$, then

\begin{align*}
\widehat{[u,k_H]}(u_x)=&T0^{H}_G(u_x)+ \frac{d}{dt}(t[u,k_H](x))|_{t=0}\\
=& T0^{H}_G(u_x)+ \frac{d}{dt}(t(D_{u}k)_H(x))|_{t=0}\\
=& \widehat{(D_uk)_H}(v_x).
\end{align*}

This finishes the proof.

\end{proof}

\subsubsection{Dirac groupoids as $\mathcal{LA}$-groupoids}

Let $L_G$ be a multiplicative Dirac structure on a Lie groupoid $G\rightrightarrows M$. This means that we have a $\mathcal{VB}$-subgroupoid $L_G\rightrightarrows E$ of $\mathbb{T}G\rightrightarrows TM\oplus A^*G$, such that $L_G\subseteq \mathbb{T}G$ is also a Dirac subbundle. In particular there is a canonical Lie algebroid structure on $L_G\rmap G$ with anchor map $L_G\rmap TG$ the natural projection and Lie bracket $\Cour{\cdot,\cdot}$ on $\Gamma_G(L_G)$. Given sections $e_1,e_2$ of $E$, there exist star sections $a_1,a_2$ of $L_G$ covering $e_1$ and $e_2$, respectively. Since $L_G$ is involutive with respect to the Courant bracket, we conclude that $\Cour{a_1,a_2}$ is a star section of $L_G$ covering a section $e$ of $E$. We define $[e_1,e_2] := e$. A straightforward computation shows that with respect to this Lie bracket and the natural projection $E\rmap TM$, the vector bundle $E\rmap M$ becomes a Lie algebroid.

\begin{proposition}
 
A multiplicative Dirac structure $L_G$ on $G$ gives rise to an $\mathcal{LA}$-groupoid

\begin{align}
\xy 
(-15,10)*+{L_G}="t0"; (-15,-10)*+{E}="b0"; 
(15,10)*+{G}="t1"; (15,-10)*+{M}="b1"; 
{\ar@<.25ex>^{p_{G}\oplus c_G} "t0"; "t1"}; 
{\ar@<.25ex>_{q_{E}} "b0"; "b1"}; 
{\ar@<.5ex> "t0"; "b0"}; 
{\ar@<-.5ex> "t0"; "b0"}; 
{\ar@<.5ex> "t1"; "b1"}; 
{\ar@<-.5ex> "t1"; "b1"}; 
\endxy
\end{align}

\noindent where $p_G$ and $c_G$ denote the tangent projection and the cotangent projection, respectively.

\end{proposition}

\begin{proof}
 
Since the structure mappings defining the Lie groupoid $L_G\rightrightarrows E$ are restrictions of the structure mappings of the tangent and cotangent groupoids, a straightforward computation shows that these structure mappings are Lie algebroid morphisms over the structure mapping of $G$. The fact that the multiplication on $L_G$ is a Lie algebroid morphism over the multiplication on $G$ follows from Proposition \ref{multiplicationLG}. An argument similar to the one used in the proof of Proposition \ref{multiplicationLG} shows that the inversion map on $L_G$ is a Lie algebroid morphism. This proves the statement.

\end{proof}

\subsection{The Lie algebroid of a multiplicative Dirac structure}\label{Lieofmultiplicativedirac}

Let $G$ be a Lie groupoid over $M$ with Lie algebroid $AG$. Let $L_G$ be a multiplicative Dirac structure on a Lie groupoid $G$. According to Proposition \ref{multiplicativepairing}, the canonical pairing $\langle\cdot,\cdot\rangle_G:\mathbb{T}G\oplus \mathbb{T}G\rmap \R$ is a Lie groupoid morphism. Applying the Lie functor yields a nondegenerate symmetric pairing 

$$A(\langle\cdot,\cdot\rangle_G):(A(TG)\oplus A(T^*G))\times_{AG}(A(TG)\oplus A(T^*G))\rmap \R.$$

\noindent Let $\langle\cdot,\cdot\rangle_{AG}$ denote the canonical nondegenerate symmetric pairing on $\mathbb{T}(AG)$. Recall that there exist canonical isomorphisms of Lie algebroids $j_G:T(AG)\rmap A(TG)$ and $j'_G:A(T^*G)\rmap T^*(AG)$ (see \eqref{ATG} and \eqref{AT*G}). Since $\langle\cdot,\cdot\rangle_{AG}$ is just a suitable restriction of $T\langle\cdot,\cdot\rangle_G$, one concludes that the canonical map 

$$j^{-1}_G\oplus j'_G:A(TG)\oplus A(T^*G)\rmap T(AG)\oplus T^*(AG),$$

\noindent is a fiberwise isometry with respect to $A(\langle\cdot,\cdot\rangle_G)$ and $\langle\cdot,\cdot\rangle_{AG}$. This is a useful tool for transporting Lagrangian subbundles of $TG\oplus T^*G$ to Lagrangian subbundles of $T(AG)\oplus T^*(AG)$. For instance, given a $\mathcal{VB}$-subgroupoid $L_G$ of $TG\oplus T^*G$, we can apply the Lie functor to obtain a $\mathcal{VB}$-subalgebroid $A(L_G)\subseteq A(TG)\oplus A(T^*G)$. We mimic the construction of tangent Dirac structures, giving rise to a $\mathcal{VB}$-subalgebroid of $T(AG)\oplus T^*(AG)$ defined by $$L_{AG}:=(j^{-1}_G\oplus j'_G)(A(L_G)).$$

\noindent The following result is a straightforward consequence of the previous discussion.

\begin{proposition}\label{multiplicativeisotropy}
 
Let $L_G\subseteq TG\oplus T^*G$ be a $\mathcal{VB}$-subgroupoid. Consider the associated $\mathcal{VB}$-subalgebroid $L_{AG}\subseteq T(AG)\oplus T^*(AG)$. Then $L_G$ is isotropic with respect to $\langle\cdot,\cdot\rangle_G$ if and only if $L_{AG}$ is isotropic with respect to $\langle\cdot,\cdot\rangle_{AG}$.

\end{proposition}

In particular, if $L_G\subseteq TG\oplus T^*G$ be a $\mathcal{VB}$-subgroupoid with associated $\mathcal{VB}$-subalgebroid $L_{AG}\subseteq T(AG)\oplus T^*(AG)$. Then $L_G$ is an almost Dirac structure on $G$ if and only if $L_{AG}$ is an almost Dirac structure on $AG$.

Now we want to deal with integrability issues. For that, consider a multiplicative Dirac structure $L_G\subseteq \mathbb{T}G$ and let $\mathcal{LA}$-groupoid $(L_G,G,E,M)$ be the associated $\mathcal{LA}$-groupoid. Applying the Lie functor we obtain the prolonged Lie algebroid structure on $A(L_G)\rmap AG$, and we use the canonical map $j^{-1}_G\oplus j'_G:A(TG)\oplus A(T^*G)\rmap T(AG)\oplus T^*(AG)$, to define a Lie algebroid $L_{AG}=(j^{-1}_G\oplus j'_G)(A(L_G))$ over $AG$, characterized by the fact that $j^{-1}_G\oplus j'_G:A(L_G)\rmap L_{AG}$ is a Lie algebroid isomorphism. We have seen that $L_{AG}\subseteq \mathbb{T}(AG)$ is a Lagrangian subbundle with respect to the canonical pairing $\langle\cdot,\cdot\rangle_{AG}$ on $\mathbb{T}(AG)$. We claim that the Lie bracket on $\Gamma_{AG}(L_{AG})$ induced by the prolonged Lie bracket on $\Gamma_{AG}(A(L_G))$ coincides with the Courant bracket. Indeed, since the tangent Lie algebroid $TL_G\rmap TG$ is isomorphic to $L_{TG}\rmap TG$, where the latter is equipped with the algebroid structure induced by the tangent Dirac structure $L_{TG}\subset TTG\oplus T^*TG$, and $A(L_G)$ is a Lie subalgebroid of $TL_G$ (Proposition \ref{natural}), then the bracket on sections of $L_{AG}$ induced by the identification $A(L_G)=L_{AG}$ is exactly the restriction of the Courant bracket on $\Gamma(T(AG)\oplus T^*(AG))$. We have proved the following.

\begin{theorem}\label{maintheorem1}

Let $G$ be a Lie groupoid with Lie algebroid $AG$. Then, there is a canonical map

\begin{align*}
\mathrm{Dir}_{mult}(G)&\rmap \mathrm{Dir}_{morph}(AG)\\
L_G&\mapsto L_{AG}:=(j^{-1}_G\oplus j'_G)(A(L_G))
\end{align*}

That is, up to a canonical identification, the Lie algebroid of a multiplicative Dirac structure $L_G\subset \mathbb{T}G$ defines a Dirac structure $L_{AG}$ on $AG$ which is also a Lie subalgebroid of $\mathbb{T}(AG)$.

\end{theorem}

It is interesting to observe that since $L_{AG}$ is the Lie algebroid of the $\mathcal{LA}$-groupoid $L_G$, in particular $L_{AG}$ inherits the structure of a \emph{double Lie algebroid} \cite{MacdoubleII}. Double Lie algebroids were introduced by Mackenzie in \cite{MacEr} as a way to understand Drinfeld's doubles of Lie bialgebroids. As a result, multiplicative Dirac structures provide interesting examples of double Lie algebroids.

\subsection{Dirac groupoids vs Dirac algebroids}

This section is concerned with the statement and proof of the main result of this work. We will prove that, whenever $G$ is a source simply connected Lie groupoid with Lie algebroid $AG$, then the map in Theorem \ref{maintheorem1} is a bijection.

For that, recall that if $M$ is a smooth manifold and $L\subset \mathbb{T}M$ is a Lagrangian subbundle, then there is a well defined element $\mu_L\in\Gamma(\wedge^3L^*)$ given by

\begin{equation}\label{couranttensor}
\mu_L(a_1,a_2,a_2):=\langle \Cour{a_1,a_2}, a_3\rangle.
\end{equation}

The element $\mu_L\in\Gamma(\wedge^3L^*)$ is referred to as the \textbf{Courant $3$-tensor} of $L$. Notice that a Lagrangian subbundle $L\subset \mathbb{T}M$ is a Dirac structure if and only if $\mu_L$ vanishes.

\begin{proposition}\label{multiplicativemu}

Let $G$ be a Lie groupoid over $M$. Consider a Lagrangian subbundle $L_G\subset TG\oplus T^*G$, which is also a Lie subgroupoid. Then, the Courant $3$-tensor of $L_G$ is multiplicative, that is

$$\mu_{L_G}:\prod_{p_G\oplus c_G}^{3}L_G\rmap \R,$$

\noindent is a groupoid morphism. 
\end{proposition}

\begin{proof}

Let us consider composable pairs $a^i_g,\overline{a}^i_h$ in $L_G$ with $i=1,2,3.$ Set $c^i_{gh}=m_{\mathbb{T}}(a^i_g,\overline{a}^i_h)\in (L_G)_{gh}$, for $i=1,2,3$. Choose a section $c^i\in\Gamma(L_G)$ such that $c^i(gh)=c^i_{gh}$. Since $L_G$ is a $\mathcal{VB}$-groupoid, the multiplication on $L_G$ is fiberwise surjective. In particular, there exist sections $a^i,\overline{a}^i\in\Gamma(L_G)$ such that $m_{\mathbb{T}}\circ (a^i,\overline{a}^i)=c^i\circ m_G$, for every $i=1,2,3$. Clearly $a^i(g)=a^i_g$ and $\overline{a}^i(h)=\overline{a}^i_h$, for $i=1,2,3$.

Then,

\begin{align*}
 \mu_{L_G}((a^1_g,a^2_g,a^3_g)*(\overline{a}^1_h,\overline{a}^2_h,\overline{a}^3_h))=& \mu_{L_G}(c^1_{gh},c^2_{gh},c^3_{gh})\\
 =& \langle \Cour{c^1,c^2}(gh), c^3(gh)\rangle\\
 =& \langle m_{\mathbb{T}}(\Cour{a^2,a^2},\Cour{\overline{a}^1,\overline{a}^2})(g,h),m_{\mathbb{T}}(a^3,\overline{a}^3)(g,h) \rangle
\end{align*}

\noindent The last identity follows from the fact that $(m_{\mathbb{T}},m_G)$ is a Courant morphism (see Proposition \ref{multiplicationLG}). Now we use the fact that $\langle\cdot,\cdot\rangle_G$ is a groupoid morphism to conclude that

$$\mu_{L_G}((a^1_g,a^2_g,a^3_g)*(\overline{a}^1_h,\overline{a}^2_h,\overline{a}^3_h))=\mu_{L_G}(a^1_g,a^2_g,a^3_g) + \mu_{L_G}(\overline{a}^1_h,\overline{a}^2_h,\overline{a}^3_h).$$

\noindent This proves that the function $\mu_{L_G}$ is multiplicative.

\end{proof}

We would like to describe explicitly the Lie algebroid morphism induced by the multiplicative tensor $\mu_{L_G}:\prod_{p_G\oplus c_G}^{3}L_G\rmap \R$. For that, we need the next lemma.

\begin{lemma}\label{prop:tangentonmu}

Let $M$ be a smooth manifold. Consider a Lagrangian subbundle $L_M\subset \mathbb{T}M$. Then, for every $(\dot{a}_1,\dot{a}_2,\dot{a}_3)\in TL_M$ the following identity holds

$$T\mu_{L_M}(\dot{a}_1,\dot{a}_2,\dot{a}_3)=\mu_{L_{TM}}((J_M\oplus\Theta_M)\dot{a}_1,(J_M\oplus\Theta_M)\dot{a}_2,(J_M\oplus\Theta_M)\dot{a}_3),$$

\noindent where $L_{TM}\subset \mathbb{T}(TM)$ is the tangent lift of $L_M$.

\end{lemma}

\begin{proof}
 
For every $a_1,a_2,a_3\in \Gamma_M(L_M)$ one has $T\mu_M(Ta_1,Ta_2,Ta_3)=T(\mu_M(a_1,a_2,a_3))$. On the other hand, the canonical map $J_M\oplus\Theta_M$ applied to each of the sections $Ta_1,Ta_2,Ta_3$ gives $a_1^T,a_2^T,a_3^T\in\Gamma_{TM}(L_{TM})$. Thus we conclude that
\begin{align*}
\mu_{TM}(a_1^T,a_2^T,a_3^T)=&\langle\Cour{a_1^T,a_2^T},a_3^T\rangle_{TM}\\
=&(\langle\Cour{a_1,a_2},a_3\rangle_M)^T, 
\end{align*}

\noindent which is exactly the tangent functor applied to the function $\mu_M(a_1,a_2,a_3)$. Therefore, for every triple of sections $a_1,a_2,a_3$ of $L_M$ we get
\begin{equation}\label{tangentonmu}
T\mu_{M}(Ta_1,Ta_2,Ta_3)=\mu_{TM}(a_1^T,a_2^T,a_3^T).
\end{equation}

 Now we notice, using local coordinates, that for every point $\dot{a}\in TL_M$ above $\dot{x}\in TM$ there exists a section $a\in\Gamma_M(L_M)$ such that $Ta(\dot{x})=\dot{a}$, where $Ta\in\Gamma_{TM}(TL_M)$ is the section obtained by applying the tangent functor to the section $a$ of $L_M$. This fact together with identity (\ref{tangentonmu}) prove the statement.
\end{proof}

As a consequence we obtain a direct proof of the Courant integrability of the tangent lift of a Dirac structure $L_M$ on $M$.

\begin{corollary}

Let $L_M$ be an almost Dirac structure on $M$, and consider the induced almost Dirac structure $L_{TM}$ on $TM$. Then $L_{TM}$ is Courant integrable if $L_M$ is Courant integrable. 

\end{corollary}

Consider now a multiplicative Dirac structure $L_G$ on $G$.  The application of the Lie functor to the groupoid morphism $\mu_{L_G}$ of Proposition \ref{multiplicativemu}, yields a Lie algebroid morphism 

$$A(\mu_{L_G}):\prod_{A(p_G\oplus c_G)}^{3}A(L_G)\rmap \R.$$

\noindent Since $A(\mu_{L_G})=T\mu_{L_G}|_{A(L_G)}$, we conclude the following.

\begin{proposition}\label{lieonmu}
 
Consider the Lagrangian subbundle $L_{AG}=(j^{-1}_G\oplus j'_G)A(L_G)\subseteq \mathbb{T}(AG)$. The following identity holds

$$A(\mu_{L_G})=\mu_{L_{AG}}\circ(j^{-1}_G\oplus j'_G)^{(3)},$$

\noindent where $(j^{-1}_G\oplus j'_G)^{(3)}:\prod_{A(p_G\oplus c_G)}^{3}A(L_G)\rmap \prod_{p_{AG}\oplus c_{AG}}^{3}L_{AG}$ denotes the natural extension of $(j^{-1}_G\oplus j'_G)$.

\end{proposition}

\begin{proof}

This follows directly from Lemma \ref{prop:tangentonmu} and the fact that $j_G$ and $j'_G$ are suitable restrictions of $J_G$ and $\Theta_G$, respectively.

\end{proof}












Now we are ready to state the main theorem of this work.

\begin{theorem}\label{maintheorem}

Let $G$ be a source simply connected Lie groupoid with Lie algebroid $AG$. There is a one-to-one correspondence between

\begin{enumerate}

\item multiplicative Dirac structures on $G$, and

\item morphic Dirac structures on $AG$.

\end{enumerate}

The correspondence is given by the map in Theorem \ref{maintheorem1}

\end{theorem}

\begin{proof}

Let $L_G$ be a multiplicative Dirac structure on $G$. Consider the Lagrangian subbundle $L_{AG}:=(j^{-1}_G\oplus j'_G)(A(L_G))\subset \mathbb{T}AG$. Since $\mu_{L_G}\equiv 0$, then Proposition \ref{lieonmu} implies that $\mu_{L_{AG}}\equiv 0$. Thus, $L_{AG}$ is a Dirac structure on $AG$ which is clearly morphic. Notice that the integrability of $L_{AG}$ is also consequence of Theorem \ref{maintheorem1}. Conversely, consider an element $L_A\in\mathrm{Dir}_{morph}(AG)$, that is $L_A$ is a linear Dirac structure on $AG$ such that $L_A\subseteq \mathbb{T}AG$ is a $\mathcal{VB}$-subalgebroid. Notice that, since $G$ is source simply connected, then $\mathbb{T}G$ is the source simply connected Lie groupoid which integrates the Lie algebroid $\mathbb{T}AG$. As explained in \cite{BursztynCabreraHoyo}, the $\mathcal{VB}$-subalgebroid $L_A\subseteq \mathbb{T}A$ integrates to a source simply connected $\mathcal{VB}$-subgroupoid $L_G\subseteq \mathbb{T}G$. We will prove that $L_G$ is a multiplicative Dirac structure on $G$. Since $L_{AG}\subseteq\mathbb{T}AG$ is Lagrangian with respect to the canonical symmetric pairing $\langle\cdot,\cdot\rangle_{AG}$ on $\mathbb{T}AG$, we conclude from Proposition \ref{multiplicativeisotropy} that $L_G$ is Lagrangian with respect to the canonical symmetric pairing $\langle\cdot,\cdot\rangle_{G}$ on $\mathbb{T}G$. It remains to show that $L_G\subseteq \mathbb{T}G$ is integrable with respect to the Courant bracket. Equivalently, we have to prove that the Courant $3$-tensor $\mu_{L_G}\in\Gamma(\wedge^3L^*_G)$ is zero. Since $L_{A}\subseteq\mathbb{T}AG$ is a Dirac structure, the induced Courant $3$-tensor $\mu_{L_{A}}\in\Gamma(\wedge^3L^*_{A})$ vanishes. Therefore, combining Proposition \ref{lieonmu} (applied to the zero Lie algebroid morphism) with Lie's second theorem we conclude that $\mu_{L_G}\equiv 0$, as desired. This shows that $L_G$ is a Dirac structure on $G$, which by definition is multiplicative.

\end{proof}

\begin{remark}
Notice that, Theorem \ref{maintheorem} provides a direct proof of the integrability of the Lagrangian subbundle $L_{AG}\subset \mathbb{T}(AG)$ associated to a multiplicative Dirac structure $L_G\subset \mathbb{T}G$, without using the theory of $\mathcal{LA}$-groupoids. In spite of this, we believe that it is interesting by itself the fact that $L_{AG}$ inherits the structure of a double Lie algebroid, which relies on the observation that $L_G$ is an $\mathcal{LA}$-groupoid.
\end{remark}

 


 



\subsection{Main examples revisited}

We have shown several examples of Dirac groupoids and Dirac algebroids. See sections \ref{DiracGroupoids}
and \ref{DiracAlgebroids}, respectively. Here we will see that both classes of examples are related by the construction explained in subsection \ref{Lieofmultiplicativedirac}.  Otherwise mentioned, throughout this subsection $G$ denotes a Lie groupoid over $M$ with Lie algebroid $AG$.

\subsubsection{\textbf{Poisson groupoids and Lie bialgebroids}}\label{poissongroupoidsrevisited}

Consider a multiplicative Poisson bivector $\pi_G$ on $G$. It is well known that in this case $M\subseteq G$ is a coisotropic submanifold and, in particular, the conormal bundle $N^*M\cong A^*G$ inherits a Lie algebroid structure.  The Dirac structure $L_G$ on $G$ defined by the graph of $\pi_G$ is a multiplicative Dirac structure. The multiplicativity of this Dirac structure is equivalent to $\pi^{\sharp}_G:T^*G\rmap TG$ being a morphism of Lie groupoids, and the associated Lie algebroid morphism coincides, up to identifications, with $\pi^{\sharp}_{AG}:T^*(AG)\rmap T(AG)$ where $\pi_{AG}$ denotes the linear Poisson bivector on $AG$ dual to the Lie algebroid $A^*G$. One concludes that the corresponding Dirac structure $L_{AG}$ on $AG$ is exactly the graph of $\pi_{AG}$.  Since $L_{AG}$ is a Lie subalgebroid of $\mathbb{T}AG$, the bundle map $\pi^{\sharp}_{AG}:T^*(AG)\rmap T(AG)$ is a Lie algebroid morphism. This is equivalent to saying that $(AG,A^*G)$ is a Lie bialgebroid. As a corollary of Theorem \ref{maintheorem} we obtain the following result.

\begin{corollary}\cite{MX2}

Let $G$ be a source simply connected Lie groupoid with Lie algebroid $AG$. There is a one-to-one correspondence between:

\begin{enumerate}

\item multiplicative Poisson bivectors on $G$, and

\item Lie bialgebroid structures on $(AG,A^*G).$

\end{enumerate}
 
\end{corollary}

\subsubsection{\textbf{Multiplicative $2$-forms and IM-$2$-forms}}

Assume that $\omega_G\in\Omega^2(G)$ is a multiplicative closed $2$-form on $G$. The Dirac structure $L_G$ given by the graph of $\omega^{\sharp}_G:TG\rmap T^*G$ is multiplicative. In this case, the corresponding Dirac structure $L_{AG}$ on $AG$ is given by the graph of the closed $2$-form $\omega_{AG}:=-\sigma^*\omega_{can}$ where $\sigma:AG\rmap T^*M$ is defined by $\sigma(u)=i_u\omega_G|_{TM}$. Since the Dirac structure $L_{AG}$ is a Lie sualgebroid of $\mathbb{T}(AG)$, we conclude that the bundle map $\omega^{\sharp}_{AG}:T(AG)\rmap T^*(AG)$ is a Lie algebroid morphism. As shown in \cite{BCO}, this is equivalent to the bundle map $\sigma:AG\rmap T^*M$ being an \textbf{IM-$2$-form} on $AG$, that is, for every $u,v\in \Gamma(AG)$, the following conditions hold

\begin{itemize}
\item $\langle \sigma(u), \rho_{AG}(v)\rangle =- \langle \sigma(v),\rho_{AG}(u)\rangle;$

\item $\sigma[u,v]=\Lie_{\rho_{AG}(u)}\sigma(v)-\Lie_{\rho_{AG}(v)}\sigma(u) + d\langle \sigma(u),\rho_{AG}(v)\rangle.$

\end{itemize}

As a corollary of Theorem \ref{maintheorem}, we get the following result.

\begin{corollary}\cite{BCWZ}

Let $G$ be a source simply connected Lie groupoid with Lie algebroid $AG$. There is a one-to-one correspondence between:

\begin{enumerate}

\item multiplicative closed $2$-forms on $G$, and

\item IM-$2$-forms on $AG$.

\end{enumerate}

\end{corollary}

\subsubsection{\textbf{Foliated groupoids and Foliated algebroids}}

Let $F_G\subseteq TG$ be a multiplicative involutive subbundle. Then, the Dirac structure $L_{G}=F_G\oplus F^{\circ}_G$ is multiplicative. The corresponding Dirac structure $L_{AG}$ on $AG$ associated to $L_G$ is given by $L_{AG}=F_{AG}\oplus F^{\circ}_{AG}\subset \mathbb{T}(AG)$, where $F_{AG}:=j^{-1}_G(A(F_G))\subseteq T(AG)$. Since $L_{AG}$ is a Dirac structure which is also a Lie subalgebroid of $\mathbb{T}(AG)$, we conclude that $F_{AG}\subseteq T(AG)$ is an involutive subbundle which is also a Lie subalgebroid of $T(AG)\rmap TM$. We refer to  such subbundle as a \textbf{morphic foliation} on $AG$.  As a corollary of Theorem \ref{maintheorem}, we obtain the next result.

\begin{corollary}\cite{Eli}

Let $G$ be a source simply connected Lie groupoid with Lie algebroid $AG$. There exists a one-to-one correspondence between:

\begin{enumerate}

\item multiplicative foliations on $G$, and

\item morphic foliations on $AG.$

\end{enumerate}

\end{corollary}

As shown in \cite{Eli, JotzOrtiz}, having a morphic foliation on $AG$ is equivalent to $AG$ be equipped with an \textbf{IM-foliation}, that is, a triple $(F_M,K,\nabla)$ where $F_M\subseteq TM$ is an involutive subbundle, $K\subseteq AG$ is a Lie subalgebroid with $\rho_{AG}(K)\subseteq F_M$, and $\nabla$ is an $F_M$-connection on $AG/K$ satisfying the following conditions

\begin{itemize}

\item $\nabla$ is flat;

\item if $u\in \Gamma(AG)$ satisfies $\nabla_{\Gamma(F_M)}(u+K)\in \Gamma(K)$, then $[u,\Gamma(K)]\subseteq \Gamma(K);$

\item if $u,v\in\Gamma(AG)$ are such that $\nabla_{\Gamma(F_M)(u+K)}, \nabla_{\Gamma(F_M)(v+K)}\in \Gamma(K)$, then $\nabla_{\Gamma(F_M)([u,v]+K)};$

\item if $u\in \Gamma(AG)$ satisfies $\nabla_{\Gamma(F_M)}(u+K)\in \Gamma(K)$, then $[\rho_{AG}(u),\Gamma(F_M)]\subseteq \Gamma(F_M).$

\end{itemize}

The properties as above determine completely the morphic foliation $F_{AG}$ on $AG$. In particular, Dirac structures of the form $L_{AG}=F_{AG}\oplus F^{\circ}_{AG}$ are in one-to-one correspondence with IM-foliations. Additionally, there exists a conceptually clear interpretation of IM-foliations in terms of Representations up to homotopy. This interpretation makes part of work in progress.

\subsubsection{\textbf{Dirac Lie groups and Dirac Lie algebras}}

Let $G$ be a Lie \emph{group} with Lie algebra $\mathfrak{g}$ and let $L_G\in\mathrm{Dir}_{mult}(G)$ be a multiplicatice Dirac structure. Consider the Dirac structure $L_{\mathfrak{g}}$ on $\mathfrak{g}$ associated to $L_G$. It was shown in \cite{Ortiz1} that $\Ker(L_G):=L_G\cap TG$ is a regular involutive subbundle of $TG$, in particular $\Ker(L_{\mathfrak{g}})=j^{-1}_G(A(\Ker(L_G)))$ is an involutive subbundle of $T\mathfrak{g}$. Since $\Ker(L_{\mathfrak{g}})$ is a linear foliation on $\mathfrak{g}$, i.e. multiplicative with respect to the abelian group structure on $\mathfrak{g}$, then the leaf through $0\in \mathfrak{g}$ is a vector subspace $\mathfrak{h}\subseteq \mathfrak{g}$. The other leaves are affine subspaces of $\mathfrak{g}$ modeled on $\mathfrak{h}$. In particular, the space of characteristic leaves of $L_{\mathfrak{g}}$ coincides with the quotient space $\mathfrak{g}/\mathfrak{h}$. The fact that $L_{g}\subseteq \mathbb{T}\mathfrak{g}$ is a Lie subalgebroid implies that $\mathfrak{h}\subseteq \mathfrak{g}$ is an ideal. Therefore, the space of characteristic leaves $\mathfrak{g}/\mathfrak{h}$ of $L_{\mathfrak{g}}$ inherits a unique Lie algebra structure making the quotient map $\phi:\mathfrak{g}\rmap \mathfrak{g}/\mathfrak{h}$ into a surjective Lie algebra morphism. Since $\mathfrak{g}/\mathfrak{h}$ is the space of characteristic leaves of $L_{\mathfrak{g}}$, there is a unique Poisson structure $\pi$ on $\mathfrak{g}/\mathfrak{h}$ making the quotient map $\phi:\mathfrak{g}\rmap \mathfrak{g}/\mathfrak{h}$ into a forward and backward Dirac map. Since $L_{\mathfrak{g}}$ is a morphic Dirac structure, we conclude that $\pi$ is a morphic bivector on $\mathfrak{g}/\mathfrak{h}$. In particular, the pair $(\mathfrak{g}/\mathfrak{h},(\mathfrak{g}/\mathfrak{h})^*)$ is a Lie bialgebra. Conversely, given a Lie algebra $\mathfrak{g}$ and an ideal $\mathfrak{h}\subseteq \mathfrak{g}$ such that $(\mathfrak{g}/\mathfrak{h},(\mathfrak{g}/\mathfrak{h})^*)$ is a Lie bialgebra, then the linear Poisson bivector $\pi$ on $\mathfrak{g}/\mathfrak{h}$ is morphic. The surjective Lie algebra morphism $\mathfrak{g}\rmap \mathfrak{g}/\mathfrak{h}$ induces a Dirac structure $L_{\mathfrak{g}}$ on $\mathfrak{g}$ (the pull back of $\pi$) which is morphic as well. We have proved the following result.

\begin{proposition}

Let $\mathfrak{g}$ be a finite dimensional Lie algebra. There is a one-to-one correspondence between:

\begin{enumerate}

\item morphic Dirac structures on $\mathfrak{g}$, and

\item ideals $\mathfrak{h}\subseteq \mathfrak{g}$ such that $(\mathfrak{g}/\mathfrak{h},(\mathfrak{g}/\mathfrak{h})^*)$ is a Lie bialgebra.

\end{enumerate}

\end{proposition}

The proposition above recovers the results of \cite{Ortiz1}.

\subsubsection{\textbf{Tangent lifts of Dirac structures}}

Let $L_G$ be a multiplicative Dirac structure on $G$. Consider the associated morphic Dirac structure $L_{AG}$ on the Lie algebroid of $G$. We can lift $L_G$ to a multiplicative Dirac structure on the tangent groupoid $TG$. Similarly, as explained in subsection \ref{tangentliftmorphic}, the morphic Dirac structure $L_{AG}$ can be lifted to a morphic Dirac structure $L_{T(AG)}$ on the tangent Lie algebroid $T(AG)\rmap TM$. It is straightforward to check that the morphic Dirac structure on $T(AG)$ associated to $L_{TG}$ as in Theorem \ref{maintheorem} coincides with the tangent lift $L_{T(AG)}$ of $L_{AG}$. That is, the tangent functor commutes with the Lie functor.

\subsubsection{\textbf{Symmetries of Dirac groupoids}}

Let $L_G$ be a multiplicative Dirac structure on $G$. Consider the associated morphic Dirac structure $L_{AG}$ on $AG$ as in Theorem \ref{maintheorem}. Let $H$ be a Lie group acting freely and properly on $G$ by groupoid automorphisms $\Phi_h:G\rmap G$, $h\in H$. Applying the Lie functor to each $\Phi_h:G\rmap G$ yields a free and proper $H$-action on $AG$ by Lie algebroid automorphisms $A(\Phi_h):AG\rmap AH$, $h\in H$. Assume that the $H$-orbits of $G$ concide with the characteristic leaves of $L_G$. Then, the $H$-orbits of $AG$ coincide with the characteristic leaves of $L_{AG}$. We have shown that in this situation we can endow the space of characteristic leaves $G/H$ of $L_G$ with a unique multiplicative Poisson bivector $\pi_{G/H}$ making the quotient map $G\rmap G/H$ into a forward and backward Dirac map. Similarly, the space of characteristic leaves $AG/H$ of $L_{AG}$ inherits a unique morphic Poisson structure $\pi_{AG/H}$ making the quotient map $AG\rmap AG/H$ into a forward and backward Dirac map.  One can easily see that the morphic Dirac structure $L_{AG/H}$ associated to $\pi_{G/H}$ as in \ref{poissongroupoidsrevisited} coincides with the morphic Dirac structure on $AG/H$ given by the graph of $\pi_{AG/H}$. As a consequence, the Lie bialgebroid of $(G/H,\pi_{G/H})$ is exactly $(AG/H,(AG/H)^*)$.

\subsubsection{\textbf{B-field transformations}}

Let $L_G$ be a multiplicative Dirac structure on $G$. Assume that $B_G$ is a multiplicative closed $2$-form on $G$. Consider the Dirac structure $L^B_G$ on $G$, obtained out of $L_G$ by applying the $B$-field transformation with respect to $B_G$. As observed in \cite{BCO}, every multiplicative closed $2$-form on $G$ induces a morphic closed $2$-form $B_{AG}$ on $AG$. A direct computation shows that the morphic Dirac structure $L^B_{AG}$ corresponding to $L^B_G$ (as in Theorem \ref{maintheorem}) is given by the $B$-field transformation of $L_{AG}$ with respect to $B_{AG}$, in agreement with \cite{Ortiz2}. 

\subsubsection{\textbf{Generalized complex groupoids}}

Let $L_G\subseteq \mathbb{T}_{\C}G$ be a multiplicative generalized complex structure on $G$. Notice that the construction explained in Theorem \ref{maintheorem1} applies also to the case of multiplicative generalized complex structures. As a result, there is a morphic Dirac structure $L_{AG}\subseteq \mathbb{T}_{\C}AG$ given by $L_{AG}:=(j^{-1}_G\oplus j'_G)_{\C}(A(L_G))$, where $(j^{-1}_G\oplus j'_G)_{\C}:A(\mathbb{T}_{\C}G)\rmap \mathbb{T}_{\C}(AG)$ denotes the complexification of the canonical isomorphism $(j^{-1}_G\oplus j'_G):A(\mathbb{T}G)\rmap \mathbb{T}(AG)$. Observe that $L_{AG}\subseteq \mathbb{T}_{\C}AG$ is in fact a generalized complex structure making the pair $(AG,L_{AG})$ into a generalized Lie algebroid. For that, we only need to check that $L_{AG}\cap \overline{L_{AG}}=\{0\}$. Indeed, one easily checks that the conjugation map $\overline{(\cdot)}_G:\mathbb{T}_{\C}G\rmap \mathbb{T}_{\C}G$ is a Lie groupoid isomorphism. Therefore, the generalized complex structure $\overline{L}_G$ on $G$ is also multiplicative. Since $\mathbb{T}_{\C}G=L_G\oplus \overline{L}_G$, the application of the Lie functor yields a decomposition

\begin{equation}\label{decompositionGCA}
A(\mathbb{T}_{\C}G)=A(L_G)\oplus A(\overline{L}_G).
\end{equation}

A straightforward computation shows that the Lie algebroid isomorphism $A(\overline{(\cdot)}_G):A(\mathbb{T}_{\C}G)\rmap A(\mathbb{T}_{\C}G)$ satisfies 

$$(j^{-1}_G\oplus j'_G)_{\C}\circ A(\overline{(\cdot)}_G)=\overline{(\cdot)}_{AG},$$

\noindent where the map of the right hand side of the identity above is the conjugation map $\mathbb{T}_{\C}(AG)\rmap \mathbb{T}_{\C}(AG)$. Hence, applying the canonical isomorphism $(j^{-1}_G\oplus j'_G)_{\C}:A(\mathbb{T}_{\C}G)\rmap \mathbb{T}_{\C}(AG)$ on both sides of \eqref{decompositionGCA}, gives rise to

$$\mathbb{T}_{\C}AG=L_{AG}\oplus \overline{L}_{AG}.$$

Therefore, $L_{AG}$ is transversal to $\overline{L}_{AG}$ and we conclude that $L_{AG}$ is a morphic generalized complex structure. In this situation, Theorem \ref{maintheorem} gives rise to the following result.

\begin{proposition}\cite{JotzStienonXu}

Let $G$ be a source simply connected Lie groupoid with Lie algebroid $AG$. There is a one-to-one correspondence between:

\begin{enumerate}

\item multiplicative generalized complex structures on $G$, and

\item morphic generalized complex structures on $AG$. 

\end{enumerate}

\end{proposition}


\section{Conclusions and final remarks}

This work can be considered as the first step to describe multiplicative Dirac structures infinitesimally. We have seen that every multiplicative Dirac structure $L_G$ on a Lie groupoid $G$ induces a Dirac structure $L_{AG}$ on its Lie algebroid $AG$ which is compatible with the algebroid structure in the sense that $L_{AG}\subseteq \mathbb{T}(AG)$ is a Lie subalgebroid. Notice that in the special situation of Poisson groupoids (resp. multiplicative closed $2$-forms, multiplicative foliations) the induced Dirac structure on $AG$ is equivalent to endow $(AG,A^*G)$ with a Lie bialgebroid structure (resp. IM-$2$-form, IM-foliation). Therefore, it would be interesting to introduce a suitable notion of \textbf{IM-Dirac structure}, providing a more explicit description of Dirac structures compatible with a Lie algebroid, unifying different infinitesimal structures such as Lie bialgebroids, IM-$2$-forms and IM-foliations. This study will be part of a future work.

\end{document}